\documentclass{amsart}	

\usepackage{soul}

\usepackage{comment}

\def\noteson{
    \gdef\note##1{\noindent{\color{blue}[##1]}}
    \gdef\todo##1{\noindent{$\dagger$ {\bfseries todo}: {\color{red}##1}}}
    \gdef\justin##1{\emph{\color{green!40!black}##1}}

}

\noteson

\RequirePackage[OT1]{fontenc}
\RequirePackage{verbatim, amscd,epsfig,amsmath,amsthm,amstext,amssymb,tikz}
\usepackage{graphicx}
\newcommand*{\horzbar}{\rule[.5ex]{2.5ex}{0.5pt}}

\usepackage{url}
\usepackage{mathrsfs}
\usepackage[all]{xy}
\usepackage{lineno}
\usepackage{tikz}
\usetikzlibrary{patterns}
\usepackage{caption}
\usepackage{subcaption}
\newcommand{\lightercolor}[3]{% Reference Color, Percentage, New Color Name
    \colorlet{#3}{#1!#2!white}
}
\newcommand{\darkercolor}[3]{% Reference Color, Percentage, New Color Name
    \colorlet{#3}{#1!#2!black}
}
\darkercolor{gray}{50}{dgray}
\lightercolor{gray}{50}{lgray}

\usepackage{bbm}
\newcommand{\cS}{\mathcal{S}}
\newcommand{\cR}{\mathcal{R}}
\newcommand{\PHT}{\operatorname{PHT}}
\newcommand{\BCT}{\operatorname{BCT}}
\newcommand{\ECT}{\operatorname{ECT}}

\newcommand{\Z}{\mathbb{ Z}}

\newcommand{\rank}{\operatorname{rank}}
\newcommand{\vol}{\operatorname{vol}}
\newcommand{\R}{\mathbb{ R}}

\newcommand{\im}{\operatorname{im}}

\newcommand{\card}{\operatorname{card}}

\newcommand\define[1]{\textbf{#1}}

\newcommand{\Omin}{\mathcal{O}}
\newcommand{\CF}{\text{CF}}
\newcommand{\Rad}{\mathcal{R}}
\newcommand{\CS}{\text{CS}}
\newcommand{\AffGr}{\text{AffGr}}
\newcommand{\Dgm}{\text{Dgm}}

\newcommand{\dom}{\operatorname{dom}}

\newcommand{\St}{\operatorname{St}}
\newcommand{\LwSt}{\operatorname{LwSt}}

\newcommand{\calB}{\mathcal{B}}

\theoremstyle{plain}
\newtheorem{thm}{Theorem}[section]
\newtheorem{lem}[thm]{Lemma}
\newtheorem{prop}[thm]{Proposition}
\newtheorem{cor}[thm]{Corollary}

\newtheorem{defn}[thm]{Definition}
\newtheorem{ex}[thm]{Example}
\newtheorem{rmk-defn}[thm]{Remark}
\newtheorem{rmk-thm}[thm]{Remark}
\newtheorem{rmk-lem}[thm]{Remark}
\newtheorem{rmk-prop}[thm]{Remark}

\numberwithin{equation}{section}

\newcommand*{\vertbar}{\rule[-1ex]{0.5pt}{2.5ex}}
\title[How Many Directions Determine a Shape?]{How Many Directions Determine a Shape \\ and other Sufficiency Results for Two Topological Transforms}

\author{Justin Curry}
\address{Department of Mathematics and Statistics, University at Albany SUNY, Albany, NY USA}
\email{jmcurry@albany.edu}

\author{Sayan Mukherjee}
\address{Departments of Statistical Science, Mathematics, Computer Science, Biostatistics \& Bioinformatics, Duke University; Durham, NC USA}
\email{sayan@stat.duke.edu}

\author{Katharine Turner}
\address{Mathematical Sciences Institute, Australian National University; Canberra, Australia}
\email{katharine.turner@anu.edu.au}

\subjclass[2010]{Primary: 46M20, 52C45; Secondary: 60B05}
\keywords{Euler calculus, persistent homology, statistical shape analysis}

\begin{document}

\begin{abstract}
In this paper we consider two topological transforms that are popular in applied topology: the Persistent Homology Transform (PHT) and the Euler Characteristic Transform (ECT).
Both of these transforms are of interest for their mathematical properties as well as their applications to science and engineering, because they provide a way of summarizing shapes in a topological, yet quantitative, way.
Both transforms take a shape, viewed as a tame subset $M$ of $\R^d$, and associates to each direction $v\in S^{d-1}$ a shape summary obtained by scanning $M$ in the direction $v$.
These shape summaries are either persistence diagrams or piecewise constant integer-valued functions called Euler curves.
By using an inversion theorem of Schapira, we show that both transforms are injective on the space of shapes, i.e.~each shape has a unique transform.
Moreover, we prove that these transforms determine continuous maps from the sphere to the space of persistence diagrams, equipped with any Wasserstein $p$-distance, or the space of Euler curves, equipped with certain $L^p$ norms.
By making use of a stratified space structure on the sphere, induced by hyperplane divisions, we prove additional uniqueness results in terms of distributions on the space of Euler curves.
Finally, our main result proves that any shape in a certain uncountable space of PL embedded shapes with plausible geometric bounds can be uniquely determined using only finitely many directions.
%This result is perhaps best appreciated in terms of shattering number or the perspective that any point in these particular moduli spaces of shapes is indexed using a tree of finite depth. \textbf{EDIT: I think the referee was very much against the shattering number interpretation in version 1, so let's perhaps stick with the old abstract?}
\end{abstract}

\maketitle

%%   %%%%%%%%  %%%%%%%%  %%%%%%%%  %%%%%%%%%%
%%   %%        %%        %%            %%
%%   %%%%%%%%  %%%%%%%%  %%            %%  
%%         %%  %%        %%            %%
%%   %%%%%%%%  %%%%%%%%  %%%%%%%%      %%
\section{Introduction}

In this paper we consider two topological transforms that are of theoretical and practical interest: the Euler Characteristic Transform (ECT) and the Persistent Homology Transform (PHT).
At a high level both of these transforms take a shape, viewed as a subset $M$ of $\R^d$, and associates to each direction $v\in S^{d-1}$ a shape summary obtained by scanning $M$ in the direction $v$.
This process of scanning has a Morse-theoretic and persistent-topological flavor---we study the topology of the sublevel sets of the height function $h_v=\langle v,\cdot\rangle|_M$ as the height varies.
These evolving sublevel sets are summarized using either the \define{Euler Curve}, which records the Euler characteristic of each sublevel set, or the \define{persistence diagram}, which pairs critical values of $h_v$ in a computational way.
In short, the ECT and the PHT are transforms that associate to any sufficiently tame subset $M\subset \R^d$ a map from the sphere to the space of Euler curves and persistence diagrams, respectively.

Before introducing the mathematical properties of these transforms, as well as the results proved in this paper, we point out some of the applications of interest.
In both data science and computational geometry, quantifying differences in shape is a difficult problem.
Part of the problem is structural: most statistical analyses operate on scalar-valued quantities, and it is very hard to summarize shape with a single number.
To address this, Euler curves and persistence diagrams are slightly more complicated summary objects endowed with metrics amenable to detecting qualitative and quantitative differences in shape.
To wit, in~\cite{PHT} the PHT was introduced and applied to the study of heel bones from primates.
By quantifying the shapes of these heel bones and clustering these using certain metrics on persistence diagrams, the authors of that paper were able to automatically differentiate species of primates using quantified differences in phylogenetic expression of heel bone shape.
A variant on the Euler Characteristic Transform was also used in~\cite{tumor} to predict clinical outcomes of brain tumors based on their shape.
Here the applicability of topological methods is particularly compelling: aggressive tumors tend to grow ``roots'' that invade nearby tissues and these protrusions are often well-detected by changes in the Euler curves.

Of course, in each of these applications, one may wonder: To what extent are these shape summaries lossy?
Topological invariants such as Euler characteristic and homology are obviously not faithful; there are non-homeomorphic spaces that appear to be identical when using these particular lenses for viewing shape.
The surprising fact developed in~\cite{Schapira:tom, PHT} and carried further in this paper and independently in~\cite{RG-ECT-inject}, is that by considering the Euler curve for every possible direction one can completely recover a shape.
Said differently, the Euler Characteristic Transform is a statistic that is sufficient to summarize compact, definable subsets of $\R^d$, see Theorem~\ref{thm:ECT-injects}.
Moreover, since homology determines Euler characteristic via an alternating sum of Betti numbers, we obtain injectivity results for the PHT as well, see Theorem~\ref{thm:PHT-injects}.

The choice of our class of shapes is an important variable throughout the paper.
We use ``constructible sets'' at first, these are compact subsets of $\R^d$ that can be constructed with finitely many geometric and logical operations.
The theory for carving out this collection of sets comes from logic and the study of o-minimal/definable structures, as they provide a way of banishing shapes that are considered ``too wild.''
One key property of working in this setting is that
every constructible set can be triangulated. This is the natural setting for Schapira's inversion result~\cite{Schapira:tom}, which is the engine that drives many of the results in our paper, specifically Theorems~\ref{thm:ECT-injects} and~\ref{thm:PHT-injects}, and independently Theorem 5 and Corollary 6 in~\cite{RG-ECT-inject}. 

An additional reason to consider o-minimal sets is that they are naturally stratified into manifold pieces.
This is important because it implies that the constructions that underly the ECT and PHT are also stratified, and that both of these two transforms can be described in terms of constructible sheaves. 
The upshot of this observation, which is developed in Section~\ref{sec:stratified}, is that the ECT of a particular shape $M$ should be determined, in some sense, by finitely many directions, certainly if $M$ is known in advance.
Indeed, any o-minimal set induces a stratification of the sphere of directions, and whenever we restrict to a particular stratum the variation of the transform can be described by homeomorphisms of the real line.
To make precise how to relate the transform for two different directions in a fixed stratum, we specialize to the case where $M=K$ is a piece-wise linearly (PL) embedded simplicial complex.
This allows us to provide an explicit formula for the ECT and construct explicit linear relations for both transforms whenever two directions lie in the same stratum.
This is the content of Lemma~\ref{lem:KxU} and Proposition~\ref{prop:deduce}.

By considering a ``generic'' class of PL embedded simplicial complexes, we leverage the explicit formula for the Euler Characteristic Transform to produce a new measure-theoretic perspective on shapes.
In Theorem~\ref{thm:align} we prove that by considering the pushforward of the Lebesgue measure on $S^{d-1}$ into the space of Euler curves (or persistence diagrams) one can uniquely determine a generic embedded simplicial complex up to rotation and reflection, i.e.~an element of $O(d)$.
The importance of this result for applications is that one can faithfully compare two un-aligned shapes simply by studying their associated distribution of Euler curves or persistence diagrams.

Finally, we extend the themes of stratification theory and injectivity results to answer our titular question, which is perhaps more accurately worded as ``How many directions are needed to infer a shape?''
Here the idea is that there is a shape hidden from view, perhaps cloaked by our sphere of directions, and we would like to learn as much about the shape as possible.
Our mode of interrogation is that we can specify a direction and an oracle will tell us the Euler curve of the shape when viewed in that direction.
Of course, by our earlier injectivity result, Theorem~\ref{thm:ECT-injects}, we know that if we query all possible directions, then we can uniquely determine any constructible shape.
However, a natural question of theoretical and practical importance is whether finitely many queries suffice.
The main result of our paper, Theorem~\ref{thm:finitenum}, shows that if we impose some a priori assumptions on our (uncountable) set of possible shapes, then finitely many queries indeed do suffice.
However, unlike the popular parlor game \emph{Twenty Questions,} the number of questions/directions we might need to infer our hidden shape is only bounded by
\[
\Delta(d,\delta,k_\delta) = \left((d-1) \, k_\delta \, \left(\frac{2\delta}{\sin(\delta)}\right)^{d-1} +1 \right) \left(1+\frac{3}{\delta}\right)^d+O\left(\frac{dk_\delta}{\delta^{d-1}}\right)^{2d}.
\]
The parameters above reflect a priori geometric assumptions that we make about our hidden shape.
Specifically, we assume our shape is well modeled as a geometric simplicial complex embedded in $\R^d$, with a lower bound $\delta$ on the ``curvature'' at every vertex, and a uniform upper bound $k_\delta$  on the number of  critical values when viewed in any given $\delta$-ball of directions. We call this class of shapes $\mathcal{K}(d,\delta,k_\delta)$.
Moreover, the two terms in the above sum are best understood as a two step interrogation procedure.
The first term counts the number of directions that we'd sample the ECT at for any shape.
The second term in the above sum counts the maximum number of questions we might ask, adapted to the answers given in the first round of questions.
%We note that there is perhaps an interesting way of interpreting our result as a \emph{shattering number} for the set of shapes $\mathcal{K}(d,\delta,k)$.
%Every question/direction ``shatters'' the set of possible shapes into two pieces and our results prove that we can uniquely specify any element of $\mathcal{K}(d,\delta,k)$ after at most $\Delta(d,\delta,k)$ number of shatters.

%%   %%%%%%%%  %%%%%%%%  %%%%%%%%  %%%%%%%%%%
%%   %%        %%        %%            %%
%%   %%%%%%%%  %%%%%%%%  %%            %%  
%%         %%  %%        %%            %%
%%   %%%%%%%%  %%%%%%%%  %%%%%%%%      %%
\section{Background on Euler Calculus}

%\color{red}
%Give a gloabl citation for o-minimal structures, such as [16], near the begining of the section.
%\color{black}

Often in geometry and topology we require that our sets and maps have certain tameness properties.
These tameness properties are exhibited by the following categories: piecewise linear, semialgebraic, and subanalytic sets and mappings.
Logicians have generalized and abstracted these categories into the notion of an o-minimal structure ~\cite{vdD}, which we now define.

\begin{defn}\label{defn:o-min}
An \define{o-minimal structure} $\Omin=\{\Omin_d\}$ specifies for each $d\geq 0$, a collection of subsets $\Omin_d$ of $\R^d$ closed under intersection and complement.
These collections are related to each other by the following rules:
\begin{enumerate}
\item If $A\in \Omin_{d}$, then $A\times \R$ and $\R \times A$ are both in $\Omin_{d+1}$; and
\item If $A\in \Omin_{d+1}$, then $\pi(A)\in \Omin_d$ where $\pi:\R^{d+1}\to\R^{d}$ is axis-aligned projection.
\end{enumerate}
We further require that $\Omin$ be closed with respect to all the operations of $\R$ that make it an ordered field, i.e. comparison $(<)$, addition $(+)$ and multiplication $(\cdot)$, and that $\Omin_1$ contain no more and no less than all finite unions of points and open intervals in $\R$.
Elements of $\Omin$ are called \define{tame} or \define{definable} sets.
Our assumptions guarantee that every semialgebraic set is definable~\cite[Ch. 2, 2.11]{vdD}.
A \define{definable map} $f:\R^d \to \R^n$ is one whose graph is definable.
We define a \define{constructible set} as a compact definable subset of $\R^d$ and denote the set of all constructible subsets of $\R^d$ by $\CS(\R^d)$.
\end{defn}

\begin{rmk-defn}
Our assumptions include more than is strictly required by van den Dries's notion of an o-minimal structure, which is defined in \cite[Ch. 1, 2.1]{vdD}, in order to contain the sets we are most interested in, i.e. semialgebraic sets and their generalizations.
What we have specified in Definition~\ref{defn:o-min} is more accurately called the o-minimal structure generated by the ordered field $(\R,<,0,1,-,+,\cdot)$ with parameters/constants in $\R$, cf.~\cite[Ch. 1, \S 5]{vdD}. 
As van den Dries describes in \cite[Ch. 1, 5.1]{vdD} and proves in \cite[Ch. 2, 2.11]{vdD} our notion of an o-minimal structure contains the collection of semialgebraic sets. 
%The collection of semialgebraic sets, sets defined in terms of polynomial inequalities, is an example of an o-minimal structure.
The advantage of o-minimal structures is that one can consider expansions of the collection of semialgebraic sets to include a restricted class of analytic functions, such as the square root function, without sacrificing desired tameness properties.
Further expansions include a notion of exponentiation and the Pfaffian, which by Wilkie's theorem provides a model-complete theory~\cite{Wilkie96}.
\end{rmk-defn}

Tame/definable sets play the role of measurable sets for an integration theory based on Euler characteristic called \define{The Euler Calculus}, see~\cite{EulerTome} for an expository review.
The guarantee that definable sets can be measured by Euler characteristic is by virtue of the following theorem.

\begin{thm}[Triangulation Theorem~\cite{vdD}]\label{thm:triangulation}
Any tame set admits a definable bijection with a subcollection of open simplices in the geometric realization of a finite Euclidean simplicial complex.
Moreover, this bijection can be made to respect a partition of a tame set into tame subsets.
\end{thm}

In view of the Triangulation Theorem, one can define the Euler characteristic of a tame set in terms of an alternating count of the number of simplices used in a definable triangulation.

\begin{defn}
If $X\in \Omin$ is tame and $h:X\to \cup \sigma_i$ is a definable bijection with a collection of open simplices, then the \define{definable Euler characteristic} of $X$ is
\[
\chi(X):=\sum_i (-1)^{\dim \sigma_{i}}
\]
where $\dim \sigma_i$ denotes the dimension of the open simplex $\sigma_i$.
We understand that $\chi(\varnothing)=0$ since this corresponds to the empty sum.
\end{defn}

As one might expect, the definable Euler characteristic does not depend on a particular choice of triangulation~\cite[pp.70-71]{vdD}, as it is a definable homeomorphism invariant.
However, as the next example shows, the definable Euler characteristic is not a homotopy invariant.

\begin{ex}
The definable Euler characteristic of the open unit interval $X=(0,1)$ is $-1$.
Note that $X$ is contractible to a point, which has definable Euler characteristic $1$.
The reader might notice that these computations coincide with the compactly-supported Euler characteristic, which can be defined in terms of the ranks of compactly-supported cohomology or Borel-Moore homology groups. However, this equivalence only makes sense for locally compact sets and not all o-minimal sets are locally compact.
\end{ex}

\begin{rmk-defn}
We will often drop the prefix ``definable'' and simply refer to ``the Euler characteristic'' of a definable set. 
\end{rmk-defn}

The definable Euler characteristic satisfies the \define{inclusion-exclusion rule}:

\begin{prop}[\cite{vdD} p.71]
For tame subsets $A,B\in\Omin$ we have
\[
\chi(A\cup B) + \chi(A\cap B) = \chi(A) + \chi(B)
\]
\end{prop}

Consequently, the definable Euler characteristic specifies a valuation on any o-minimal structure, i.e.~it serves the role of a measure without the requirement that sets be assigned a non-negative value.
This allows us to develop an integration theory called Euler calculus that is well-defined for so-called constructible functions, which we now define.

\begin{defn}
A \define{constructible function} $\phi:X \to \Z$ is an integer-valued function on a tame set $X$ with the property that every level set is tame and only finitely many level sets are non-empty.
The set of constructible functions with domain $X$, denoted $\CF(X)$, is closed under pointwise addition and multiplication, thereby making $\CF(X)$ into a ring.
\end{defn}

% \begin{rmk}
% Note that if we post-compose $\phi: X\to \Z$ with the inclusion $\Z \hookrightarrow \R$, then a constructible function specifies a definable map from $X\subseteq \R^d$ to $\R$.
% The requirement that only finitely many level-sets 
% \end{rmk}

\begin{defn}
The \define{Euler integral} of a constructible function $\phi:X \to \Z$ is the sum of the Euler characteristics of each of its level-sets, i.e.
\[
\int \phi \, d\chi := \sum_{n=-\infty}^{\infty} n \cdot \chi(\phi^{-1}(n)).
\]
Note that constructibility of $\phi$ implies that only finitely many of the $\phi^{-1}(n)$ are non-empty so this sum is well-defined. 
\end{defn}

Like any good calculus, there is an accompanying suite of canonical operations in this theory---pullback, pushforward, convolution, etc.

\begin{defn}
Let $f: X \to Y$ be a tame mapping between between definable sets.
Let $\phi_Y: Y \to \Z$ be a constructible function on $Y$.
The \define{pullback} of $\phi_Y$ along $f$ is defined pointwise by
\[
f^*\phi_Y(x)=\phi_Y(f(x)).
\]
The pullback operation defines a ring homomorphism $f^*:\CF(Y) \to \CF(X)$.
\end{defn}

The dual operation of pushing forward a constructible function along a tame map is given by integrating along the fibers. 

\begin{defn}
The \define{pushforward} of a constructible function $\phi_X:X \to \Z$ along a tame map $f:X\to Y$ is given by
\[
f_*\phi_X(y)=\int_{f^{-1}(y)} \phi_X d\chi.
\]
This defines a group homomorphism $f_*: \CF(X) \to \CF(Y)$.
\end{defn}

Putting these two operations together allows one to define our first topological transform: the Radon transform.

\begin{defn}
Suppose $S\subset X \times Y$ is a locally closed definable subset of the product of two definable sets.
Let $\pi_X$ and $\pi_Y$ denote the projections from the product onto the indicated factors.
The \define{Radon transform with respect to $S$} is the group homomorphism $\Rad_S : \CF(X) \to \CF(Y)$ that takes a constructible function on $X$, $\phi:X \to \Z$, pulls it back to the product space $X\times Y$, multiplies by the indicator function of $S$ before pushing forward to $Y$. 
In equational form, the Radon transform is
\[
\Rad_S (\phi):=\pi_{Y*} [(\pi_X^*\phi) 1_S].
\]
\end{defn}

The following inversion theorem of Schapira~\cite{Schapira:tom} gives a topological criterion for the invertibility of the transform $\Rad_S$ in terms of the subset $S\subset X \times Y$.

\begin{thm}[\cite{Schapira:tom} Theorem 3.1]\label{thm:inversion}
If $S\subset X\times Y$ and $S'\subset Y\times X$ have fibers $S_x$ and $S'_x$ in $Y$ satisfying
\begin{enumerate}
	\item $\chi(S_x\cap S'_x)=\chi_1$ for all $x\in X$, and
	\item $\chi(S_x\cap S'_{x'})=\chi_2$ for all $x'\neq x \in X$,
\end{enumerate}
then for all $\phi\in \CF(X)$,
\[
(\Rad_{S'} \circ \Rad_{S})\phi = (\chi_1-\chi_2)\phi +\chi_2 \left(\int_X \phi d\chi\right) 1_X .
\]
\end{thm}

In the next section we show how to use Schapira's result to deduce the injectivity properties of the Euler Characteristic and Persistent Homology Transforms.

%%   %%%%%%%%  %%%%%%%%  %%%%%%%%  %%%%%%%%%%
%%   %%        %%        %%            %%
%%   %%%%%%%%  %%%%%%%%  %%            %%  
%%         %%  %%        %%            %%
%%   %%%%%%%%  %%%%%%%%  %%%%%%%%      %%
\section{Injectivity of the Euler Characteristic Transform}\label{sec:ECT-injects}

In the previous section we introduced background on definable sets, constructible functions and the Radon transform.
We now specialize this material to the study of persistent-type topological transforms on definable subsets of $\R^n$.
What makes these transforms ``persistent'' is that they study the evolution of topological invariants with respect to a real parameter.
In this section we begin with the simpler of these two invariants, the Euler characteristic.

\begin{defn}
The \define{Euler Characteristic Transform} takes a constructible function $\phi$ on $\R^d$ and returns a constructible function on 
%the - extra "the" found by referee in item 1
$S^{d-1}\times \R$ whose value at a direction $v$ and real parameter $t\in\R$ is the Euler integral of the restriction of $\phi$ to the half space $x\cdot v \leq t$.
In equational form, we have
\[
\ECT: \CF(\R^d) \to \CF(S^{d-1} \times \R) \qquad \text{where} \qquad 
\ECT(\phi)(v,t):=\int_{x\cdot v \leq t} \phi \, d\chi.
\]
For our applications of interest, it suffices to restrict this transform to compact definable subsets of $\R^d$ (which we call constructible sets), where we identify a definable subset $M\in \Omin_d$ with its associated constructible indicator function $\phi=1_M$.
When we fix $v\in S^{d-1}$ and let $t\in\R$ vary, we refer to $\ECT(M)(v,-)$ as the \define{Euler curve} for direction $v$.
This allows us to equivalently view the Euler Characteristic Transform for a fixed $M\in \CS(\R^d)$ as a map from the sphere to the space of Euler curves (constructible functions on $\R$). 
\[
	\ECT(M): S^{d-1} \to \CF(\R)  \qquad \text{where} \qquad 
\ECT(M)(v)(t):=\chi(M\cap \{x \mid x\cdot v \leq t\}).
\]
The proof that the Euler curves are constructible is detailed in Lemma~\ref{lem:constructible-ECT}.
We write $\ECT(M)(v)(\infty)$ as shorthand for the limit of $\ECT(M)(v)(t)$ as $t$ goes to positive infinity. 
This limit exists by compactness and definability of $M$. 
Note that for every direction $v$ this limit is the total Euler characteristic of $M$, i.e. $\ECT(M)(v)(\infty) = \chi(M)$.
\end{defn}

\begin{rmk-defn}\label{rmk:compact-omin}
The restriction to constructible sets $\CS$ permits us to ignore the difference between ordinary and compactly-supported Euler characteristic. This is because the intersection of a compact definable set $M$ and the closed half-space $\{x\in \R^d \mid x\cdot v \leq t\}$ is compact and definable, and these two versions of Euler characteristic agree on compact sets.
This restriction is not strictly necessary, but would require a reworking of certain aspects of persistent homology (reviewed in Section~\ref{sec:PHT-injects}), which is beyond the scope of this paper.
\end{rmk-defn}

The reader may wish to pause to consider why the ECT associates to a definable set $M$, viewed as the constructible function $\phi=1_M$, a constructible function on $S^{d-1}\times \R$. This, along with other finiteness concerns, is addressed by the following lemma, which in turn depends on two versions of the Trivialization Theorem of o-minimal geometry.

\begin{thm}[The Trivialization Theorems of~\cite{vdD}]\label{thm:trivialization-theorems}
If $S\subseteq \R^n$ and $A\subseteq \R^m$ are definable subsets and $f:S\to A$ is a continuous definable map, then the \textbf{Trivialization Theorem}~\cite[Ch. 9, 1.2]{vdD} states that there is a finite partition of $A$ into finitely many definable subsets $A_1,\ldots, A_M$ such that over each $A_i$ there is a definable subset $F_i\subseteq \R^N$, for some $N$, a definable map $\lambda_i:f^{-1}(A_i) \to F_i$, such that the product map $(f_i,\lambda_i)$, where $f_i:=f|_{f^{-1}(A_i)}$, is a homeomorphism making the following diagram commute:
\[
	\xymatrix{f^{-1}(A_i) \ar[rr]^{(f_i,\lambda_i)} \ar[rd]_{f_i} & & A_i\times F_i \ar[ld]^{\pi_1} \\ & A_i & }
\]
Moreover, the \textbf{Trivialization Theorem with Distinguished Subsets}~\cite[Ch. 9, 1.7]{vdD} states that if $S$ has definable subsets $S_1,\ldots, S_k$, then we can refine the partition of $A$ to guarantee that for each $A_i$ there exists definable subsets $F_1,\ldots, F_k$ of $F$ with $(f_i,\lambda_i)(f^{-1}(A_i)\cap S_j)=A_i\times F_j$. In other words, we can refine the fiber to respect the distinguished subsets of $S$.
\end{thm}

%\begin{rmk-defn}[Constructibility]\label{rmk:constructible-ECT}
\begin{lem}[Finitely Many Topological Changes]\label{lem:constructible-ECT}%% TO BE RENAMED!!
For any definable set $M\subseteq \R^d$ and any direction $v\in S^{d-1}$ the topological, i.e. homeomorphism, type of the sublevel sets $M_{v,t}=M\cap \{x \mid x\cdot v \leq t\}$ can only change finitely many times as a function of $t$. Moreover there is a uniform bound $\kappa_M \geq 0$, depending on $M$, on the number of changes in the topological type when viewed in any direction. 
As a consequence, for any direction $v\in S^{d-1}$ the Euler curve $\ECT(M)(v,-)$ is a constructible function over $\R$.
%and Betti curves (to be defined later) are
 %Additionally, there is a uniform bound on the number of points in the persistence diagram for the sublevel set filtration of $M$ when viewed in any direction.
\end{lem}
\begin{proof}
One can associate to a definable $M\subseteq \R^d$ another definable set
\[
	X_M:=\{(x,v,t)\in M\times S^{d-1}\times \R \mid x\cdot v\leq t\}.
\]
To see that $X_M$ is definable we note that $S^{d-1}$ is the set of points satisfying 
$$x_1^2+\cdots x_d^2=1$$ making it an algebraic set, which by \cite[Ch. 2, 2.11]{vdD} is definable.
By taking products we can conclude that $\R^d\times S^{d-1} \times \R$ is definable as well.
The collection of half-spaces $H:=\{(x,v,t) \in \R^d\times S^{d-1}\times \R \mid x\cdot v \leq t\}$ is also semialgebraic and hence definable.
Consequently, the intersection of $M\times S^{d-1}\times \R$ and $H$, which is $X_M$, must be definable as well.
Finally, by our Definition~\ref{defn:o-min} we know the projection map onto the second and third factors $\pi_{2,3}:\R^d\times \R^d\times \R \to \R^d\times \R$ is definable and 
hence, by~\cite[Ch. 1, 2.3(ii)]{vdD}, the restriction of this projection map to $X_M$, i.e.
\[
	\pi: X_M \to S^{d-1}\times \R \qquad \text{where} \qquad \pi^{-1}(v,t)=M_{v,t}
\]
is a continuous definable map as well.

By applying the Trivialization Theorem~\cite[Ch. 9, 1.2]{vdD} to $\pi$ we know that there is a finite partition of $S^{d-1}\times \R$ into definable sets $A_1,\ldots, A_n$ over which the map $\pi:X_M \to S^{d-1}\times \R$ is definably trivial, i.e. the map restricts to a trivial fiber bundle over each $A_i$.
The trivialization over the sets $A_i$ guarantees that the topological type of the sublevel sets $M_{v,t}$ changes only finitely many times when viewed as a function of $v$ and $t$, but the statement of the lemma asks for finiteness when viewing the shape $M$ in a fixed direction $v$.
To obtain the statement we want, we view the $A_i$ as distinguished subsets of $S^{d-1}\times \R$ and then apply the Trivialization Theorem with Distinguished Subsets~\cite[Ch. 9, 1.7]{vdD} to the projection map $\pi_1: S^{d-1}\times\R \to S^{d-1}$.
This provides us with a partition of the sphere $S^{d-1}$ into definable subsets $B_1,\ldots, B_m$ so that over each $B_i$ there are $n$ definable subsets $F_{i1},\ldots, F_{in}$ of $\R$ so that the trivialization carries each $\pi_1^{-1}(B_i)\cap A_j$ to $B_i\times F_{ij}$.
We note that each $F_{ij}$ need not be connected, but each $F_{ij}$ will be a finite union of singleton sets and connected intervals, i.e.~definable $0$ and $1$-cells.
The number of changes in the topology of $M_{v,t}$ for any $v\in B_i$ is then bounded by the sum of the number of $0$ and $1$-cells in the partition $F_{i1},\ldots, F_{in}$ of $\R$.
Taking the maximum of this number over $B_i,\ldots, B_m$ provides the uniform bounded $\kappa_M$ on the number of topological changes when viewed in any direction.
%Notice that the above map is definable because of the properties of the class of definable sets, which implies that the fibers are definable and vary over a definable base.
\end{proof}
%\end{rmk-defn}

This proof of Lemma~\ref{lem:constructible-ECT} provides an important alternative characterization of the Euler Characteristic Transform: The ECT is simply the pushforward of the indicator function of $X_M$ along the definable map $\pi$, which defines a constructible function on $S^{d-1}\times \R$.

We now use Schapira's inversion theorem, recalled as Theorem~\ref{thm:inversion}, to prove our first injectivity result.

%\color{red}
%The reviewer comments about the use of terms "definable" and "constructible" in point 9 but I don't understand the problem. Did someone else already fix this?
%\color{black}

\begin{thm}\label{thm:ECT-injects}
Let $\CS(\R^d)$ be the set of constructible sets, i.e.~compact definable sets.
The map $\ECT: \CS(\R^d) \to \CF(S^{d-1} \times \R)$ is injective.
Equivalently, if $M$ and $M'$ are two constructible sets that determined the same association of directions to Euler curves, then they are, in fact, the same set.
Said symbolically:
\[
\ECT(M)=\ECT(M'): S^{d-1} \to \CF(\R) \qquad \Rightarrow \qquad M=M'
\]
\end{thm}
\begin{proof}
Let $M\in \CS(\R^d)$. 
Let $W$ be the hyperplane defined by $\{x\cdot v = t\}$.
By the inclusion-exclusion property of the definable Euler characteristic
\begin{eqnarray*}
\chi(M \cap W) &=& \chi(\{ x\in M: x\cdot v=t\})\\
& = &\chi(\{ x\in M: x\cdot v\leq t\} \cap \{ x\in M: x\cdot (-v) \leq -t\} )\\
& = &\chi(\{ x\in M: x\cdot v\leq t\}) + \chi( \{ x\in M: x\cdot (-v) \leq -t\} ) - \chi(M)\\
& = &\ECT(M)(v,t) + \ECT(M)(-v,-t) - \ECT(M)(v)(\infty)
\end{eqnarray*}
This means that from $\ECT(M)$ we can deduce the $\chi(M \cap W)$ for all hyperplanes $W\in \AffGr_d$.

Let $S$ be the subset of $\R^d \times \AffGr_d$ where $(x,W)\in S$ when $x$ is in the hyperplane $W$. 
For simplicity, we denote the projection to $\R^d$ by $\pi_1$ and the projection to $\AffGr_d$ by $\pi_2$.
For this choice of $S$ the Radon transform of the indicator function $1_M$ for $M \in \Omin_d$ at $W\in \AffGr_d$ is 
\begin{eqnarray*}
(\Rad_{S}1_M)(W)&=& (\pi_2)_* [(\pi_1^*1_M)1_S](W) \\
&=&\int_{(x,W)\in S} (\pi_1^* 1_M) \, d\chi\\
&=&\int_{x\in M \cap W} (\pi_1^* 1_M) \, d\chi\\
&=&\chi(M \cap W)
\end{eqnarray*}
This implies that from $\ECT(M)$ we can derive $\cR_\cS(1_M)$.

Similarly let $S'$ be the subset of $\AffGr_d \times \R^d$ where $(W,x)\in S'$ when $x$ is in the hyperplane $W$.
For all $x\in \R^d$, $S_x\cap S'_x$ is the set of hyperplanes that go through $x$ and hence is $S_x\cap S'_x=\R P^{d-1}$ and $\chi(S_x\cap S'_x)=\frac{1}{2}(1+(-1)^{d-1})$.  
For all $x\neq x'\in \R^d$, $S_x\cap S'_{x'}$ is the set of hyperplanes that go through $x$ and $x'$ and hence is $S_x\cap S'_{x'}=\R P^{d-2}$ and $\chi(S_x\cap S'_{x'})=\frac{1}{2}(1+(-1)^{d-2})$.  
Applying Theorem~\ref{thm:inversion} yields
$$(\Rad_{S'}\circ \Rad_{S})(1_M)=(-1)^{d-1} 1_M + \frac{1}{2}(1+(-1)^{d-2}) \chi(M)1_{\R^d}.$$

Note that if $\ECT(M)=\ECT(M')$ then $\Rad_{S} 1_M=\Rad_{S} 1_{M'}$, since $\ECT$ determines the Euler characteristic of every slice.
Moreover, if $\ECT(M)=\ECT(M')$, then $\chi(M)=\chi(M')$, which by inspecting the inversion formula above further implies that $1_M=1_{M'}$ and hence $M=M'$.
\end{proof}

%%   %%%%%%%%  %%%%%%%%  %%%%%%%%  %%%%%%%%%%
%%   %%        %%        %%            %%
%%   %%%%%%%%  %%%%%%%%  %%            %%  
%%         %%  %%        %%            %%
%%   %%%%%%%%  %%%%%%%%  %%%%%%%%      %%
\section{Injectivity of the Persistent Homology Transform}\label{sec:PHT-injects}

The primary transform of interest for this paper is the Persistent Homology Transform, which was first introduced in~\cite{PHT} and was initially defined for a PL embedded simplicial complex in $\R^d$.
The reader is encouraged to consult~\cite{PHT} (and~\cite{Elevate} for a related precursor) for a more complete treatment of the PHT in that setting. Here, we illustrate how this transform can be defined for any constructible set $M\in \CS(\R^d)$.

As already noted, given a direction $v\in S^{d-1}$ and a value $t\in\R$, the sublevel set $M_{v,t}:=\{x\in M \mid x\cdot v\leq t\}$ is the intersection of the constructible set $M$ with a closed half-space.
This intersection has various topological summaries, one of them being the (definable) Euler Characteristic $\chi$.
One can also consider the (cellular) homology with field coefficients $H_k$, which is defined in each degree $k\in \{0,1,...,n\}$. 
These are vector spaces that summarize topological content of any suitably nice topological space $X$, which in this paper will always be spaces of the form $M_{v,t}$.
In low degrees the interpretation of these homology vector spaces for a space $X$ are as follows:
$H_0(X)$ is a vector space with basis given by connected components, $H_1(X)$ is a vector space spanned by ``holes'' or closed loops that are not the boundaries of embedded disks, $H_2(X)$ is a vector space spanned by ``voids'' or closed two-dimensional (possibly self-intersecting) surfaces that are not the boundaries of an embedded three dimensional space.
The higher homologies $H_k(X)$ are understood by analogy: these are vector spaces spanned by closed (i.e.~without boundary) $k$-dimensional subspaces of $X$ that are themselves not the boundaries of $k+1$-dimensional spaces.
The dimension of $H_k(X)$ is called the \define{Betti number} $\beta_k(X)$, which for $X\subseteq \R^d$ always satisfy $\beta_k(X) = 0$ for $k \geq d$.

The proof that ordinary
 Euler characteristic is a topological invariant is best understood via homology.
Indeed, the Betti numbers determine the Euler characteristic via an alternating sum:
\[
	\chi(X)=\beta_0(X) - \beta_1(X) + \beta_2(X) - \beta_3(X) + \cdots
\]
However, one feature that homology enjoys that the Euler characteristic does not is \emph{functoriality}, which is the property that any continuous transformation of spaces $f:X\to Y$ induces a linear transformation of homology vector spaces $f_k : H_k(X) \to H_k(Y)$ for each degree $k$.
This is the key feature that defines sublevel set persistent homology.

\begin{defn}
Let $M\in \CS(\R^d)$ be a compact definable set. 
For any vector $v\in \R^d$ define the \define{height function in direction $v$},  $h_v(x)=\langle v,x \rangle$ as the restriction of the inner product $\langle v,\cdot\rangle$ to points $x\in M$.
The \define{sublevel set persistent homology group} in degree $k$ between $s$ and $t$ is 
\[
	PH_k(M,h_v)(s,t)=\im \iota^{s\to t}_k: H_k(M_{v,s}) \to H_k(M_{v,t}).
\]
\end{defn}

The remarkable feature of persistent homology is that one can encode the persistent homology groups for every pair of values $s\leq t$ using a finite number of points in the extended plane.
This is done via the persistence diagram.

\begin{defn}
Let $\R^{2+}$ be the part of the extended plane that is above the diagonal, i.e.~$\R^{2+}:=\{(b,d) \in \{-\infty\} \cup \R)\times (\R\cup\{\infty\}) \mid b \leq d\}$.
The \define{persistence diagram} in degree $k$ associated to $M\in \CS(\R^d)$ when filtered by sublevel sets of $h_v(x)=\langle v,x \rangle$ is the \emph{unique} finite multi-set of points $\calB_k=\{\left((b_i^k,d_i^k);n_i\right)\} \subset \R^{2+}$ with the property that for every pair of values $s\leq t$ the following equality holds
\[
\dim PH_k(M,h_v)(s,t) := \rank \iota^{s\to t}_k = \card \{PH_k(M,h_v) \cap [-\infty,s]\times(t,\infty]\}.
\]
Since the the persistence diagram $\calB_k$ completely encodes the ranks of the sublevel set persistent homology groups $PH_k(M,h_v)$, we will pass between these two notations freely, depending on what needs to be emphasized.
We further note that the persistence diagram is also called the \define{barcode}, which relies on the interpretation of each point $(b,d)\in \R^{2+}$ as an interval $I=[b,d)\subseteq \R$ where $b=\inf I$ and $d=\sup I$.
\end{defn} 

\begin{rmk-defn}[Persistence and o-Minimality]\label{rmk:omin-pers}
The existence of the persistence diagram is a non-trivial result and depends on certain tameness properties of the maps $\iota^{s\to t}_k$.
Finite dimensionality of all of the homology vector spaces $H_k(M_{v,t})$ suffices, but in the o-minimal setting things are even better behaved because the dimension of the persistent homology groups can only change finitely times.
Recall that the Triangulation Theorem, Theorem~\ref{thm:triangulation}, guarantees that for each $v$ and $t$ $M_{v,t}$ has a finite decomposition into cells and so, the cellular homology is finite.
The fact that the topology of $M_{v,t}$ can only change finitely many times as a function of $t$ was proved in Lemma~\ref{lem:constructible-ECT}.
%To see this we can modify the construction described in Lemma~\ref{lem:constructible-ECT}
%\[
%\pi: X_{M,v} \to \R \qquad \text{where} \qquad X_{M,v}:=\{(x,t)\in M\times\R \mid x\cdot v\leq t \}
%\]
%and $\pi$ is projection onto the second factor.
%Sublevel sets of $h_v(x)$ are now encoded as fibers of the map $\pi$ and The Trivialization Theorem~\cite[p.7]{vdD} implies that the topological type of the fiber of this map can only change finitely many times.
\end{rmk-defn}

To consider persistence diagrams $PH_k(M,h_v)$ associated to different directions $v\in S^{d-1}$ we consider the set of all persistence diagrams, which can be topologized in several ways, as explained after the following definition.

\begin{defn}
\define{Persistence Diagram Space}, written \text{Dgm}, is the set of all possible countable multi-sets of $\R^{2+}:=\{(b,d) \in (\{ - \infty\} \cup \R)\times (\R\cup\{\infty\}) \mid b \leq d\}$ where the number of points of the form $(b, \infty)$ and $(-\infty, d)$ are finite and $\sum_{d-b<\infty} d-b <\infty$.
Points of the form $(b,\infty)$ or $(-\infty,d)$ are called \define{essential classes} and points of the form $(b,d)$ where neither coordinate is $\infty$ are called \define{inessential classes}.
For persistence diagrams that encode the sublevel set persistent homology $PH_k(M,h_v)$ of a constructible set $M$ there are no points of the form $(-\infty,d)$.
\end{defn}

%%NOTA BENE: Made strict inequality for DGM not strict to allow for the diagonal.

We have used the term ``space'' with the implication that there is a topology on the set of all persistence diagrams. 
Indeed this topology comes from various choices of metrics on the set of persistence diagrams, which are phrased in terms of matchings of points between two persistence diagrams. 
Since a persistence diagram is technically a multiset, we append an additional coordinate to serve as a labelling index so that each $\calB$ can be regarded as a genuine set.

\begin{defn}\label{def:Wasserstein-p-distance}
%Recall that we regard a point $(b,d)\in \R^{2+}$ as an interval $I=[b,d)$, where $b=\inf I$ and $d=\sup I$.
Suppose $\mathcal{B}=\{(I;j) \mid (I;j) \in \R^{2+}\times \mathbb{N}\}$ and $\mathcal{B}'=\{(I';j) \mid (I';j) \in \R^{2+}\times \mathbb{N}\}$
are two persistence diagrams, viewed as sets rather than multisets; i.e. $(I;j)$ is to be interpreted as the $j^{th}$ copy of the interval $I$.
A \define{matching} is a partial bijection $\sigma:\mathcal{B}\to\mathcal{B'}$, i.e. a choice of subset $\mathcal{M}\subseteq \calB$, called the domain $\dom(\sigma)$, and an injection $\sigma:\mathcal{M}\to\calB'$.
We write the complement of the domain of $\sigma$ as $\dom^c(\sigma):=\calB-\dom(\sigma)$ and the complement of the image of $\sigma$ as $\im^c(\sigma):=\calB'-\im(\sigma)$; collectively these are called the \define{unmatched points} of $\sigma$.
%If two finite persistence diagrams have different cardinalities, we can always add points to the diagonal, i.e. points of the form $(b,b)$, to make them of the same size.

For this paper we will always promote a partial bijection $\sigma$ to an actual bijection via the introduction of diagonal images.
For a point $I=(b,d)\in \R^{2+}$ where neither coordinate is $\infty$, we define the \define{diagonal image of $I$} as $\Delta(I)=(\frac{b+d}{2},\frac{b+d}{2})$.
Associated to any partial bijection $\sigma:\calB\to\calB'$ is an actual bijection $\tilde{\sigma}:\calB(\sigma)\to \calB'(\sigma)$ where
\[
 \calB(\sigma):=\dom(\sigma) \, \cup \, \dom^c(\sigma) \, \cup \bigcup_{(I',j')\in\im^c(\sigma)} (\Delta(I');j')
\]
and
\[
\calB'(\sigma):=\im(\sigma)\, \cup\, \im^c(\sigma)\, \cup \bigcup_{(I,j)\in\dom^c(\sigma)} (\Delta(I);j). 
\]
The map $\tilde{\sigma}$ now matches points that were previously unmatched by $\sigma$ with their corresponding diagonal images.
For ease of notation, we drop the superscript tilde on $\tilde{\sigma}$ and simply write $\sigma$ for the extended map $\calB(\sigma)\to\calB'(\sigma)$.

The \define{$(p,q)$-cost} of a matching $\sigma$, where $\sigma_j(I)$ denotes the $\R^{2+}$ coordinates of $\sigma(I;j)$, is
\[
    W_{p,q}(\sigma)=\left(\sum_{(I,j)\in\calB(\sigma)} ||I-\sigma_j(I)||_q^p\right)^{1/p}.
\]
As a reminder, the $\ell^q$ distance between two matched points $(I;j)=(b_j,d_j;j)$ and $\sigma(I,j)=(b_k,d_k;k)$ is
\[
    ||I-\sigma_j(I)||_q = \left(|b_j-b_k|^q+|d_j-d_k|^q\right)^{1/q}
\]
with the understanding that if $d_j=\infty$, then we must have $d_k=\infty$ and this distance collapses to $|b_j-b_k|$. 
The $\ell^{\infty}$ distance is $||I-\sigma_j(I)||_{\infty}=\max\{|b_j-b_k|,|d_j-d_k|\}$.

The \define{Wasserstein $(p,q)$-distance} between two diagrams $\mathcal{B}$ and $\mathcal{B}'$ is then the infimum of this matching cost over all matchings, i.e.
\[
W_{p,q}(\mathcal{B},\mathcal{B}'):=\inf_{\sigma:\mathcal{B}\to\mathcal{B}'} W_{p,q}(\sigma).
\]
As noted in \cite{Alg-Wasserstein}, for fixed $p$, the Wasserstein $(p,q)$-distances are all bi-Lipschitz equivalent. The convention set by \cite{Wasserstein} was to refer to the Wasserstein $(p,\infty)$-distance as \define{the Wasserstein $p$-distance $W_p$}, but there are good reasons to adopt $q=p$ or $q=1$ as conventions as well. Unless explicitly stated to the contrary, our default assumption is that $q=\infty$.

We note that the Wasserstein $\infty$-distance is also called the \define{bottleneck distance}, for which we reserve the special notation
\[
    d_B(\calB,\calB'):=W_{\infty}(\calB,\calB') = \inf_{\sigma:\calB\to\calB'} \max_{(I,j)\in\calB(\sigma)}||I-\sigma_j(I)||_{\infty}.
\]
\end{defn}

Although the bottleneck distance is the preferred distance for many theoretical purposes, in~\cite{PHT} the $1$-Wasserstein distance was used. 
For more details about the geometry of the space of persistence diagrams under $p$-Wasserstein metrics see~\cite{Median}.

\begin{defn}
The \define{Persistent Homology Transform} $\PHT$ of a constructible set $M\in\CS(\R^d)$ is the map $\PHT(M): S^{d-1} \to \text{Dgm}^d$ that sends a direction $v\in S^{d-1}$ to the persistent diagrams gotten by filtering $M$ in the direction of $v$, recording one diagram for each homological degree $0\leq k \leq d-1$, i.e.
\[
	\PHT(M): v \mapsto (PH_0(M,h_v),PH_1(M,h_v), \ldots, PH_{d-1}(M,h_v)).
\]
Letting the set $M$ vary gives us the map
\[
	\PHT : \CS(\R^d) \to C(S^{d-1},\text{Dgm}^d)
\]
where $C(S^{d-1},\text{Dgm}^d)$ is the set of continuous functions from $S^{d-1}$ to $\text{Dgm}^d$, the latter being equipped with some Wasserstein $p$-distance. 
\end{defn}

Before moving on with the remainder of the paper, we offer a sheaf-theoretic interpretation of the Persistent Homology Transform, which is not necessary for the remainder of the paper.
The reader that is uninterested in sheaves can safely ignore the following remark.

\begin{rmk-defn}[Sheaf-Theoretic Definition of the PHT]\label{rmk:sheaf-PHT}
Extending Lemma~\ref{lem:constructible-ECT}, we know that associated to any constructible set $M$ is a space
\[
X_M:=\{(x,v,t)\in M\times S^{d-1} \times \R \mid x\cdot v\leq t\}
\]
and a map 
$\pi:X_M \to  S^{d-1} \times \R$ whose fiber over $(v,t)$ is the sublevel set $M_{v,t}$.
The \define{derived Persistent Homology Transform} is the right derived pushforward of the constant sheaf on $X_M$ along the map $\pi$, written $R\pi_* k_{X_M}$.
The associated cohomology sheaves $R^i\pi_*k_{X_M}$ of this derived pushforward, called the Leray sheaves in~\cite{tda-cosheaves}, has stalk value at $(v,t)$ the $i^{th}$ cohomology of the sub-level set $M_{v,t}$.
If we restrict the sheaf $R^i\pi_*k_{X_M}$ to the subspace $\{v\}\times \R$, then one obtains a constructible sheaf that is equivalent to the persistent (co)homology of the filtration of $M$ viewed in the direction of $v$.
The persistence diagram in degree $i$ is simply the expression of this restricted sheaf in terms of a direct sum of indecomposable sheaves.
\end{rmk-defn}

\subsection{Continuity of the PHT}

It should be noted that when $M$ is an embedded simplicial complex, continuity of the resulting map $\PHT(M): S^{d-1} \to \text{Dgm}^d$ was proved as Lemma 2.1 of~\cite{PHT}.
To ensure that the above definition generalizes to constructible sets, we must prove a generalization of that lemma.

\begin{lem}[cf. Lemma 2.1~\cite{PHT}]\label{lem:PHT-cts}
For a constructible set $M\in \CS(\R^d)$, the map $\PHT(M):S^{d-1} \to \text{Dgm}^d$ is continuous, where $S^{d-1}$ is given the Euclidean distance and $\text{Dgm}$ uses any Wasserstein $p$-distance
with $1 \leq p \leq \infty$.
\end{lem}
\begin{proof}
First we note that since $M$ is compact, there is a bound $D_M$ on the distance from any point in $M$ to the origin.
This implies that for any two directions $v_1$ and $v_2$, the height functions $h_{v_1}$ and $h_{v_2}$ have a point-wise bound given by
\[
	|h_{v_1}(x)-h_{v_2}(x)|=|x\cdot v_1 - x\cdot v_2|\leq ||x||_2 \cdot ||v_1-v_2||_2\leq D_M||v_1-v_2||_2
\]
Consequently, we have that $||h_{v_1}-h_{v_2}||_{\infty} \leq D_M||v_1-v_2||_2$.

The bottleneck stability theorem of~\cite{Bottleneck} guarantees that for each homological degree $k\geq 0$ the bottleneck distance between the persistence diagrams is bounded by the $L^{\infty}$ distance between the functions. 
It suffices to prove continuity of $\PHT(M)$ in each coordinate, so without loss of generality, we refer to the $k^{th}$ persistence diagram for $h_{v_1}$ as $\mathcal{B}_1$ and the $k^{th}$ persistence diagram for $h_{v_2}$ as $\mathcal{B}_2$.
The bottleneck stability theorem guarantees that
\[
	d_B(\mathcal{B}_1, \mathcal{B}_2) \leq ||h_{v_1}-h_{v_2}||_{\infty} \leq D_M ||v_1-v_2||_2.
\]
This implies continuity of $\PHT(M)$ when using the bottleneck distance on persistence diagrams, which is the Wasserstein $\infty$-distance, because in order to make the left hand side less than $\epsilon$, we need only bound the Euclidean distance $||v_1-v_2||_2$ by $\delta=\epsilon/D_M$. 

%\textbf{JUSTIN SAYS: Need to copy the proof from Edelsbrunner and Mileyko to show this works precisely.}

For general $p$ we use Lemma~\ref{lem:constructible-ECT}, which showed that for a constructible set $M$ there is a bound $\kappa_M$ on the number of critical values when viewed in any direction $v$.
%The proof that there is such a bound $\kappa_M$ follows from the fact that the map $X_{M} \to S^{d-1}\times \R$ is definable and by Theorem \ref{thm:cts-definable-strat}, we know the codomain of this map can be partitioned into finitely many strata, so let $N$ equal the number of strata.
As a consequence of this lemma, the number of points in the persistence diagrams $\mathcal{B}_1$ and $\mathcal{B}_2$ are both bounded above by $\kappa_M$, since each point in the diagram corresponds to a topological change when filtering $M$ in the direction $v_i$. 
Let $\sigma:\mathcal{B}_1\to \mathcal{B}_2$ be a matching that realizes the bottleneck distance $\epsilon:=d_B(\mathcal{B}_1,\mathcal{B}_2)$, so that in particular $||I-\sigma_j(I)||_{\infty}\leq \epsilon$ for all $(I,j)\in \calB_1(\sigma)$.
We note that it's possible that $\sigma$ matches every point in $\calB_1$ to the diagonal, which would lead to an augmented diagram $\calB_1(\sigma)$ of cardinality at most $2\kappa_M$.
Since the Wasserstein $p$-distance infimizes the $(p,\infty)$-cost over all matchings, we can now say that
\[
W_p(\mathcal{B}_1,\mathcal{B}_2)\leq \left(\sum_{(I,j)\in\mathcal{B}_1(\sigma)} ||I-\sigma_j(I)||_{\infty}^p\right)^{1/p} \leq \left( 2\kappa_M \epsilon^p\right)^{1/p} = (2\kappa_M)^{1/p}\epsilon.
\]
Since we can control $\epsilon$ in terms of $||v_1-v_2||_2$, this proves continuity of the PHT for the Wasserstein $p$-distance for $p\in[1,\infty]$.
\end{proof}

% \begin{rmk-defn}[Continuity]
% It should be noted that when $M$ is an embedded simplicial complex, continuity of the resulting map $\PHT(M): S^{d-1} \to \text{Dgm}^d$ was proved as Lemma 2.1 of~\cite{PHT}.
% To generalize this result to constructible sets
% Note that we know that these functions are continuous because of bottleneck stability combined with the observation that, for any fixed $M$, there is a uniform bound on the number of critical points of $h_v:M\to\R$; this can be deduced from Remark~\ref{rmk:constructible-ECT}. The argument is the same as that for the piecewise linear $M$ as proved in~\cite{PHT}. 
% \end{rmk-defn}

% \color{red}
% Do we need to justify that only finitely many critical values occur when $M$ is constructible? Have we done so already and I just  can't find where?
% \color{green}
% I could prove this, but I'm not sure it's as simple as we'd like...On second thought I'll just say you can deduce it from rmk:constructible-ECT.
% \color{black}

\subsection{Continuity of Betti Curves and Euler Curves}

The goal of this section will be the statement and proof of a persistent analog of the classical result that homology determines the Euler characteristic via Betti numbers.
This persistent analog will also feature the extension of the continuity result of Lemma \ref{lem:PHT-cts} to the Betti Curve Transform (Definition \ref{defn:BCT}), which will in turn imply continuity of the Euler Characteristic Transform.
Each of these results will require the specification of a metric on constructible functions over the real line, i.e. $\CF(\R)$, which we now do.

\begin{defn}\label{defn:CF-Lp}
For every $1\leq p < \infty$ we define the $L^p$ extended pseudo-metric (distance) on constructible functions as follows: For $f,g\in \CF(\R)$ we set 
\[
	d_p(f,g) = \left( \int |f(t) -g(t)|^p dt \right)^{1/p}
\]
when this integral exists and $\infty$ when it does not.
Similarly, we define the $L^{\infty}$ distance to be $d_{\infty}(f,g):=\sup_{t\in \R} |f(t)-g(t)|$ when the right hand side exists and $\infty$ when it does not.
\end{defn}

We use this distance to prove the following preparatory lemma, which details the passage from persistence diagrams to constructible functions.

\begin{lem}[Continuity of the Diagram to Function Map]\label{lem:dgm-to-CF}
Recall that every persistence diagram $\mathcal{B}=\{(I,j)\}$ can be viewed as a multiset of intervals in $\R$.
Let $\Phi(\mathcal{B})$ be the constructible function associated to $\calB$ that sums the indicator functions supported on each interval $I$ appearing in $\calB$, i.e.~$\Phi(\mathcal{B}) = \sum_{(I,j)\in \mathcal{B}} \mathbbm{1}_I$.
Let $\Dgm_{\kappa}$ denote the subset of persistence diagram space with fewer than $\kappa$ off-diagonal points.

For every $p\in[1,\infty)$, the map $\Phi:\Dgm_{\kappa}\to\CF(\R)$ is continuous when $\Dgm_{\kappa}$ is equipped with the Wasserstein $q$-distance and $\CF(\R)$ is equipped with the $L^p$ distance of Definition \ref{defn:CF-Lp}.
Additionally, this distance can be bound in terms of the bottleneck distance as follows
\[
d_p(\Phi(\calB)),\Phi(\calB'))^p \leq 2 \kappa^p d_B(\calB,\calB')
\]
\end{lem}
\begin{proof}
Since $W_q(\calB, \calB')\geq d_B(\calB,\calB')$ for all $q\in [1,\infty)$ to show that that $\Phi$ is continuous when $\Dgm_{\kappa}$ is equipped with the Wasserstein $q$-distance it is sufficient to show it is continuous with respect to the bottleneck distance.  

We note that we must exclude $p=\infty$ from the statement of the lemma, because a persistence diagram $\mathcal{B}$ is distance zero from the same diagram where a point is added on the diagonal, e.g. $\mathcal{B}\cup (b,b)$. In this case, the constructible functions $\Phi(\mathcal{B})$ and $\Phi(\mathcal{B}\cup (b,b))=\Phi(\mathcal{B})+\mathbbm{1}_{b}$ are distance one away in the $L^{\infty}$ distance, thus proving that continuity of this map for $p=\infty$ is not possible.

For $p\ne \infty$, we remark that the map $\Phi$ is intuitively continuous because a small variation in points in the diagram $\mathcal{B}$ results in a small variation in the endpoints of the simple functions making up $\Phi(\mathcal{B})$. 
By construction we know that
\[
d_p(\Phi(\calB),\Phi(\calB'))^p=\int | \sum_{(I,j)} \mathbbm{1}_{I} - \sum_{(I',j')} \mathbbm{1}_{I'}|^p
\]
To quantify this small variation precisely, let $\sigma:\calB\to\calB'$ be a matching that realizes the $(1,1)$-Wasserstein distance $W_1(\calB,\calB')$, where we work with the $\ell^1$ distance between points in the persistence diagram. 
We note that since we assume the number of points in $\calB$ and $\calB'$ are both less than $\kappa$, then this infimum must be realized. 
Without loss of generality, we can assume that two off-diagonal points will be matched to each other whenever matching to their diagonal images has the same, or greater, cost.
%Let $\hat{I}, \sigma_j(\hat{I}$ denote the pair of intervals under the matching $\sigma$ which are the furthest apart.

Consider the function $f=\sum_{(I,j)} \mathbbm{1}_I - \sum_{(I',j')} \mathbbm{1}_{I'}$. To bound $\int |f|^p$ we will use bounds on $\int |f|$ and $f$. Let $I \Delta \sigma_j(I)$ denote the symmetric difference between intervals $I$ and $\sigma_j(I)$. 
As already observed in \cite[Prop. 1.2]{Bubenik-Scott-Stanley}, we can bound the integral of $|f|$ by 
\[
\int |f|\leq\int |\sum_{(I,j)} \mathbbm{1}_{I\Delta \sigma_j(I)}|= \sum_{(I,j)} \int|\mathbbm{1}_{I\Delta \sigma_j(I)}|
= W_1(\calB, \calB').
\]
%Original equation above
% \begin{align*}
% \int |f|& = \int |\sum_{(I,j)} \mathbbm{1}_{I\Delta \sigma_j(I)}|\\
% &\leq  \sum_{(I,j)} \int|\mathbbm{1}_{I\Delta \sigma_j(I)}|\\
% &= W_1(\calB, \calB').
% \end{align*}

We also claim that $0\leq |f(x)|\leq \kappa$ for all $x$. 
To see this, first observe that there are at most $\kappa$ points in each of $\calB$ and $\calB'$. 
If $f(x)>\kappa$ then there must be intervals $I\in \calB$ and $I'\in \calB'$ such that $x\in I\cap I'$ and both $I$ and $I'$ are matched to the diagonal under $\sigma$. 
However, the cost of matching both $I$ and $I'$ to the diagonal is $\mu(I)+\mu(I')$, where $\mu$ denotes the length of each interval, whereas the cost of matching $I$ to $I'$ is $\mu(I\Delta I')\leq \mu(I)+\mu(I')$ as $I$ and $I$ intersect.
This contradicts our assumption that $\sigma$ was an optimal matching for $W_{1,1}$.

Using the bound $0 \leq f(x) \leq \kappa$, the integral $\int |f|^p$ is bounded above by the integral $\int |g|^p$ where $\int |g|=\int |f|$ and $g$ only takes on the values of $\kappa$ or zero. That is 
\[
\int |f|^p \leq \kappa^p\frac{\int|f|}{\kappa}=\kappa^{p-1}\int|f|
\]
and hence \[d_p(\Phi(\calB),\Phi(\calB'))^p\leq \kappa^{p-1} W_1(\calB, \calB').\]

To finish the proof we only need to observe that if $\calB, \calB'\in \Dgm_{\kappa}$ then both $W_1(\calB, \calB')\leq 2\kappa W_q(\calB, \calB')$ (for $q\in [1,\infty)$)and $W_1(\calB, \calB')\leq 2\kappa d_B(\calB, \calB')$.  

\end{proof}

%We now specialize the above result to consider persistence diagrams and their associated constructible functions that come from viewing a constructible set $M$ in multiple directions. 
Recall that for fixed $M\in\CS(\R^d)$ and $v\in S^{d-1}$ the Betti number in degree $k$ of the sublevel set $M_{v,t}$ is
\[
\beta_k(M_{v,t})=\dim H_k(M_{v,t}).
\]
For each $t\in\R$, the right hand side of the above equation is determined by the cardinality of the intersection of the persistence diagram $PH_k(M,h_v)$ with the half-open quadrant $[-\infty,t]\times(t,\infty]$.
However, by allowing $t$ to vary, one obtains the \define{Betti curve} $\beta_{k,v}(t)$ for the direction $v$ in degree $k$ as the sum of indicator functions supported on the intervals appearing in the persistence diagram for $PH_k(M,h_v)$, i.e. $\Phi(PH_k(M,h_v))$ where $\Phi$ was defined in Lemma \ref{lem:dgm-to-CF}.
These observations provide a third topological transform that stands between the PHT and the ECT.

\begin{defn}[Betti Curve Transform]\label{defn:BCT}
Let $M$ be a constructible set and let $\PHT(M):S^{d-1}\to\Dgm^d$ be the associated persistent homology transform.
Applying the map $\Phi:\Dgm\to\CF(\R)$ in each coordinate yields the \define{Betti Curve Transform} (BCT)
\[
	\BCT(M): v \mapsto (\beta_{0,v}(t),\beta_{1,v}(t),\ldots,\beta_{d-1,v}(t)).
\]
% Letting the set $M$ vary gives us the map
% \[
% 	\BCT : \CS(\R^d) \to C(S^{d-1},\CF(\R)^d)
% \]
% where $C(S^{d-1},\CF(\R)^d)$ is the set of continuous functions from $S^{d-1}$ to $\CF(\R)^d$, the latter being equipped with some $L^p$-distance for $p\in[1,\infty)$. 
\end{defn}

We now prove that the Betti Curve Transform is continuous as a function of the viewing direction $v\in S^{d-1}$. 
Although this is an immediate corollary of Lemma \ref{lem:dgm-to-CF}, we will provide a more concrete ``for all $\epsilon$,
there exists a $\delta$'' proof using the Bottleneck distance bound proved earlier.

\begin{cor}[Continuity of Betti Curves]\label{cor:betti-curve-stability}
Fix a constructible set $M$. 
For each $p\in[1,\infty)$, the Betti Curve Transform $\BCT(M):S^{d-1}\to \CF(\R)^d$ is continuous, where each coordinate $\BCT(M)_k$ is viewed as a map from the sphere to $\CF(\R)$, the latter using the $L^p$ distance on constructible functions.
\end{cor}
\begin{proof}
It suffices to prove continuity in each coordinate, which puts us in the setting of Lemma \ref{lem:dgm-to-CF}.
As was proved in Lemma~\ref{lem:constructible-ECT}, we can bound the number of off-diagonal points in any persistence diagram $PH_k(M,h_v)$ by $\kappa_M$.
This serves as the value $\kappa$ required by Lemma \ref{lem:dgm-to-CF}.
Since Lemma \ref{lem:PHT-cts} proved that the diagrams are continuous as a function of sphere direction, Lemma \ref{lem:dgm-to-CF} proves that the resulting constructible functions also vary continuously. This completes the proof.

To see a more detailed special case of this argument, we note that bottleneck stability proves that for any pair of directions $v_1$ and $v_2$ the bottleneck distance between the associated diagrams is bounded as $d_B(\calB_1,\calB_2)\leq D_M||v_1-v_2||_2$.
Lemma \ref{lem:dgm-to-CF} proves that the $L^p$ distance between the associated Betti curves is bounded as $d_p(\Phi(\calB_1),\Phi(\calB_2))\leq (2^{p+1}\kappa_M^p d_B(\calB_1,\calB_2))^{1/p}$.
Consequently to make $d_p(\Phi(\calB_1),\Phi(\calB_2))\leq \epsilon$ it suffices to make $||v_1-v_2||\leq \delta$ where $\delta=\frac{\epsilon^p}{2\kappa_M^p D_M}$.

% It will be convenient to express the Betti curve of $M$ viewed in direction $v$---ignoring the superfluous homological index $k$---as a sum $\beta_v(t)=e_v(t) + f_v(t)$, where $e_v(t)$ is the sum of indicator functions supported on essential classes, i.e. those points in the persistence diagram of the form $(b,\infty)$, and $f_v(t)$ is the sum of indicator functions supported on finite length, inessential features, i.e. points in the diagram where neither coordinate is $\infty$.
% Said more succinctly,
% \[
% \beta_v(t)=e_v(t)+f_v(t) \qquad\text{where}\qquad e_v(t)=\sum_i \mathbbm{1}_{[b_i,\infty)} \quad \text{ and } \quad f_v(t)=\sum_j \mathbbm{1}_{[b_j,d_j)}.
% \]
% Following the proof of Lemma~\ref{lem:PHT-cts} it is easy to see that if $||v_1-v_2||_2\leq \epsilon / D_M$, then
% \[
% d_p(e_{v_1},e_{v_2})^p \leq \kappa_M \epsilon^p \qquad \text{and} \qquad d_p(f_{v_1},f_{v_2})^p \leq \kappa_M \epsilon^p.
% \]
% Combining this with the usual inequality
% \[
% \int |\beta_{v_1}(t) -\beta_{v_2}(t)|^p \leq \int |e_{v_1}(t)-e_{v_2}(t)|^p + \int |f_{v_1}(t)-f_{v_2}(t)|^p \leq 2\kappa_M\epsilon^p
% \]
% yields the desired result.
% \textbf{NEED TO REWRITE THESE LAST FEW LINES TO USE THE PREVIOUS LEMMA}.
\end{proof}

%\textbf{TODO: PROVE CONTINUITY OF ECT!}

\begin{prop}\label{prop:PHT-ECT}
The Persistent Homology Transform ($\PHT$) determines the Euler Characteristic Transform ($\ECT$), i.e. we have the following commutative diagram of maps
\[
	\xymatrix{ & C(S^{d-1},\text{Dgm}^d) \ar[d]^{\alpha \circ \beta \circ \mbox{\textendash} } \\ \CS(\R^d) \ar[ur]^{\PHT} \ar[r]_-{\ECT} & C(S^{d-1},\CF( \R))}
\]
where $\beta:=\Phi^d:\Dgm^d\to\CF(\R)^d$ is the map that takes the $\PHT$ to the Betti Curve Transform $\BCT$.
The map $\alpha:\CF(\R)^d\to \CF(\R)$ takes the alternating sum of $d$ constructible functions pointwise.
Finally, we recall that the set $C(S^{d-1},\CF(\R))$ refers to the set of continuous maps from $S^{d-1}$, equipped with the restriction of the Euclidean norm, to $\CF(\R)$, equipped with any $L^p$ norm so long as $p\in [1,\infty)$; again the case $p=\infty$ is intentionally excluded, since continuity there is impossible.
%is continuous, the map from $S^{d-1}$ to $\CF(\R)$
%The proposition asserts that ECT(M) is continuous. I had to spend some time before I realised that this meant that it was continuous when interpreted as a map from S^{d-1} to CF(R), where CF(R) was being given a Wasserstein p-distance with p \ne \infty. Ok, so something like this needs to be written out in full, in the theorem statement.
%Moreover, since for any constructible set $M$ we have that $\ECT(M)$ is the composition $\alpha\circ\beta\circ \PHT(M)$ of continuous functions, then $\ECT(M)$ is continuous as well.
\end{prop}
\begin{proof}
% The persistence diagram $PH_k(M,h_v)$ determines the homology of the sublevel set $H_k(M_{v,t})$ by the cardinality of the intersection of $PH_k(M,h_v)$ with the half-open quadrant $[-\infty,t]\times(t,\infty]$.
% This is equal to $\dim PH_k(M,h_v)(t,t)$, which is $\dim H_k(M_{v,t})=\beta_k(M_{v,t})$.
% We note that for varying $t$, the function $\beta_k(M_{v,t})$ defines the $k^{th}$ \define{Betti curve} in the directon $v$.
% Performing this construction for each homological degree defines a map $\beta:\text{Dgm}^d \to \CF(\R)^d$.
% This is continuous because one may view each point in the persistence diagram as an interval and the Betti curve is simply the sum of all the indicator functions of these intervals---perturbing the endpoints of the interval results in a small perturbation of the associated constructible function (when using some $L^p$ norm \textbf{EDIT: Or I'm wrong and this needs to Wasserstein $p$-norm.}).

Corollary ~\ref{cor:betti-curve-stability} proves that each Betti curve is continuous as a function of $v\in S^{d-1}$.
This proof used, Lemma \ref{lem:dgm-to-CF}, which proved that the map from $\Dgm\to \CF(\R)$ that takes a persistence diagram to the sum of indicator functions supported on each interval is continuous when the domain is equipped with a Wasserstein $p$-norm and $\CF(\R)$ uses the $L^p$ norm.

%It only remains to be shown that 
The alternating sum of $d$ constructible functions, i.e. $\alpha:\CF(\R)^d\to \CF(\R)$, is continuous and so
%We write $d_{p,1}$ for the $\ell^1$ vector distance on $\CF(\R)^d$ with the $L^p$ distance being used in each coordinate. 
%With this choice of distance, $\alpha$ is a Lipschitz-1 map.
%Indeed, if we write $\underline{f}=(f_0,\ldots, f_{d-1})$ for a $d$-vector of construtible functions, then
%\begin{eqnarray*}
%d_p(\alpha(\underline{f}),\alpha(\underline{g}))^p & = & \int |\alpha(\underline{f})-\alpha(\underline{g})|^p \\
%& = & \int | \sum_{k=0}^{d-1} (-1)^k f_{k}(t) - \sum_{k=0}^{d-1} (-1)^k g_{k}(t)|^p \\
%& = & \int | \sum_{k=0}^{d-1} (-1)^k (f_{k}(t) - g_k(t))|^p \\
%& \leq & \sum_{k=0}^d \int | f_{k}(t) - g_k(t)|^p \\
%& = & d_{p,1}(\underline{f},\underline{g}).
%\end{eqnarray*}
%Combining all of these arguments, we have now proved that the integer-valued function
\[
	\alpha\circ\beta\circ \PHT(M)(v,t)= \sum_{k=0}^n (-1)^k \beta_k(M_{v,t})=\chi(M_{v,t}),
\]
which is $\ECT(M)(v,t)$, is continuous as a function of $v$ for any $L^p$ norm so long as $p\ne \infty$.
\end{proof}

% Recall that we proved in Lemma \ref{lem:dgm-to-CF} that the map from persistence diagrams (equipped with the Wasserstein $p$-norm) to the space of constructible functions (equipped with some $L^p$ norm) is continuous so long as $p\in[1,\infty)$.
% This was used to prove that Betti curves are continuous as a function of direction $v\in S^{d-1}$ in Corollary \ref{cor:betti-curve-stability}.
% Finally, by taking alternating sums, we proved in Proposition \ref{prop:PHT-ECT} that $\ECT(M):S^{d-1} \to \CF(\R)$ is continuous as a function of the viewing direction.
When our constructible set $M$ is a PL embedded geometric complex $K$, then we can say more about the continuity and image of the Euler Characteristic Transform, as the next remark indicates.

\begin{rmk-thm}\label{rmk:Hausdorff-ECs}
When $M$ is a PL embedded simplicial complex $K$, we will show in Proposition~\ref{prop:lower-star} that, for any direction $v$, the Euler curve is right continuous.
More precisely, the function space that $\ECT$ maps to consists of finite sums of indicator functions on half open intervals of the form $[a,b)$ where $b=\infty$ is allowed.
If we endow the space of indicator functions with an $L^p$ norm for finite $p$, then we get that if two Euler curves $f$ and $g$ differ at a filtration parameter $t$, then there is some $\epsilon >0$ so that $(f-g)|_{[t,t+\epsilon)}\neq 0$, which in turn implies that $d_p(f,g) > \epsilon$.
This implies that in this setting, the Euler Characteristic Transform lands in a Hausdorff subspace of constructible functions $\CF(\R)$ equipped with some $L^p$ norm.
\end{rmk-thm}

Finally, we remind the reader of the sheaf-theoretic and $K_0$-theoretic perspective on the Euler Characteristc Transform.

\begin{rmk-prop}[Grothendieck Group Interpretation]\label{rmk:sheaf-PHT-ECT}
Continuing Remark~\ref{rmk:sheaf-PHT}, the reader familiar with the Grothendieck group of constructible sheaves (see~\cite{Schapira:OpsOnCFs} for a clear and concise treatment) will note that Proposition~\ref{prop:PHT-ECT} coheres with the statement that the image of the (derived) Persistent Homology Transform in $K_0$ is the Euler Characteristic Transform.
\end{rmk-prop}

\subsection{Injectivity of the PHT and BCT}

We can combine Proposition~\ref{prop:PHT-ECT} with Theorem~\ref{thm:ECT-injects} to obtain a generalization of an injectivity result proved in~\cite{PHT} for simplicial complexes in $\R^2$ and $\R^3$.

\begin{thm}\label{thm:PHT-injects}
Let $\CS(\R^d)$ be the set of constructible sets, i.e.~compact definable subsets of $\R^d$. The Persistent Homology Transform $\PHT: \CS(\R^d) \to C(S^{d-1}, \text{Dgm}^d)$ and Betti Curve Transform $\BCT:\CS(\R^d) \to C(S^{d-1}, \CF(\R)^d)$ are both injective.
\end{thm}
\begin{proof}
If two constructible sets $M$ and $M'$ have $\PHT(M)=\PHT(M')$ or $\BCT(M)=\BCT(M')$, then they have $\ECT(M)=\ECT(M')$, which by Theorem \ref{thm:ECT-injects} implies that $M=M'$.
\end{proof}

\begin{rmk-thm}[Co-Discovery]
In the final stages of preparing the first version of this article, a pre-print~\cite{RG-ECT-inject} by Rob Ghrist, Rachel Levanger, and Huy Mai appeared independently proving Theorem~\ref{thm:PHT-injects}. 
The authors of that paper and this paper want to make clear that these results were independently discovered.
\end{rmk-thm}

%%   %%%%%%%%  %%%%%%%%  %%%%%%%%  %%%%%%%%%%
%%   %%        %%        %%            %%
%%   %%%%%%%%  %%%%%%%%  %%            %%  
%%         %%  %%        %%            %%
%%   %%%%%%%%  %%%%%%%%  %%%%%%%%      %%
\section{Stratified Space Structure of the ECT and PHT}\label{sec:stratified}

% One of the essential observations of this paper is that for a constructible set $M$, the topological summaries provided by the ECT or PHT exhibit tame (or constructible) behavior.
% This was the content of Lemma \ref{lem:constructible-ECT}, which proved that we could partition the space of directions and filtration values $S^{d-1}\times \R$ into finitely many regions over which the topological types of the sublevel sets $M_{v,t}$ do not change.
% By further projecting these regions onto $S^{d-1}$ and using the Trivialization Theorem with Distinguished Subsets (Theorem \ref{thm:trivialization-theorems}) one obtains a finite partition of the sphere $S^{d-1}$ over which the filtration types are related in a controlled way.

Recall that the content of Lemma \ref{lem:constructible-ECT} was two-fold: (1) that for a constructible set $M$ we could partition the space of directions and filtration values $S^{d-1}\times \R$ into finitely many regions over which the topological type of the sublevel sets $M_{v,t}$ do not change, and (2) that we could use this partition to obtain a partition of the sphere of directions so that when two directions $v$ and $w$ are in the same connected component of this partition, the persistent homology induced by $h_v$ and $h_w$ are equivalent in a certain sense.
In this section we introduce the language of stratification theory to give a slightly tighter description of these decompositions in order to make issues of dimensionality and differentiability more transparent.

In particular, by using stratification theory we can refine the partitions of Lemma \ref{lem:constructible-ECT} into connected manifolds of varying dimension called \emph{strata}.
This will allow us to make certain arguments in terms of top-dimensional strata on the sphere $S^{d-1}$.
When we eventually specialize our discussion to those $M$ that are PL embedded geometric complexes in $\R^d$, we prove that the relationship between the persistent topology for $h_v$ and $h_w$, when $v$ and $w$ are in the same stratum, is essentially linear.
These observations will be critical in concluding that only finitely many directions are needed to infer a shape.

We now define what we mean by a stratified space structure.
We note that there are many notions of a stratification, perhaps the most famous being the one due to Whitney~\cite{Whitney}.
Our notion of a stratification and of a stratified map will be a combination of Whitney's condition with definability requirements.
We follow the presentation in~\cite{o-min-strat} in order to use the results proved there.

\begin{defn}
Suppose $S_{\alpha}$ and $S_{\beta}$ are a pair of $C^1$ submanifolds of $\R^n$ with 
%$S_{\alpha} \subseteq \overline{S_{\beta}} \setminus S_{\beta}$.
$S_{\alpha} \subseteq \overline{S_{\beta}}$.
We let $T_x S_{\alpha}$ refer to the tangent space to $S_{\alpha}$ at $x$.
The pair $S_{\alpha}$ and $S_{\beta}$ satisfy \define{Whitney's condition $(b)$ at $x$} if
\begin{quote}
	for every sequence $\{x_n\}$ in $S_{\alpha}$ converging to $x$ and sequence $\{y_m\}$ in $S_{\beta}$ also converging to $x$ where the tangent spaces $T_{y_k}S_{\beta}$ converge to $\tau$ and the secant lines $\ell_k:=\langle x_k-y_k \rangle$ converge to $\ell$ then $\ell\subset \tau$.
\end{quote}
\end{defn}

This condition can be made to cohere with sets that are definable in some o-minimal structure.

\begin{defn}
A \define{definable $C^p$ stratification of $\R^n$} is a partition $\mathcal{S}$ of $\R^n$ into finitely many subsets, called strata, such that:
\begin{enumerate}
	\item Each stratum is a connected and definable $C^p$ submanifold of $\R^n$.
	\item The Axiom of the Frontier holds, i.e.~if $S_{\alpha}\cap \overline{S_{\beta}}\neq \varnothing$, then $S_{\alpha}\subseteq \overline{S_{\beta}}$.
	%For every $S_{\gamma}\in \mathcal{S}$, we have that $\overline{S_{\gamma}}\setminus S_{\gamma}$ is a union of strata in $\mathcal{S}$.
\end{enumerate}
We note that \cite[Proposition 2.2.20]{IH-Book} proves that the Axiom of the Frontier implies that the set of strata forms a partially ordered set (where $S_{\alpha}\leq S_{\beta}$ if and only if $S_{\alpha}\subseteq \overline{S_{\beta}}$)
and that the frontier of any stratum, i.e. $\overline{S_{\gamma}}\setminus S_{\gamma}$, is a union of strata of strictly lower dimension. 

A \define{definable $C^p$ Whitney stratification} is a definable $C^p$ stratification $\mathcal{S}$ such that for all $S_\alpha\leq S_{\beta}$ in $\mathcal{S}$ the pair satisfies Whitney's condition $(b)$ at all points $x\in S_{\alpha}$.
\end{defn}

We will be making use of two theorems, which say that stratifications and stratified maps can be made to respect pre-existing definable subsets.
To be precise, we require the following definition.

\begin{defn}
We say that $\mathcal{S}$ is \define{compatible} with a class $\mathcal{A}$ of subsets of $\R^n$ if each $A\in \mathcal{A}$ is a finite union of some strata in $\mathcal{S}$.
\end{defn}

\begin{thm}[cf. Thm. 1.3 of~\cite{o-min-strat}]\label{thm:compatible-strat}
If $A_1,\cdots, A_k$ are definable sets in $\R^n$, then there exists a definable $C^p$ Whitney stratification of $\R^n$ compatible with $\{ A_1,\cdots, A_k\}$.
\end{thm}

\begin{defn}
Let $f: X \to Y$ be a definable map. 
A \define{$C^p$ stratification of $f$} is a pair $(\mathcal{X},\mathcal{Y})$, where $\mathcal{X}$ and $\mathcal{Y}$ are definable $C^p$ Whitney stratifications of $X$ and $Y$ respectively.
Moreover we require that for each $X_{\alpha} \in \mathcal{X}$ there be a $Y_{\alpha}\in \mathcal{Y}$, such that $f(X_{\alpha})\subset Y_{\alpha}$ and $f|_{X_{\alpha}} : X_{\alpha} \to Y_{\alpha}$ is a $C^p$ submersion.
\end{defn}

\begin{thm}[cf. Thm.  2.2 of~\cite{o-min-strat}]\label{thm:cts-definable-strat}
Let $f:X \to Y$ be a continuous definable map.
If $\mathcal{A}$ and $\mathcal{B}$ are finite collections of definable subsets of $X$ and $Y$, respectively, then there exists a $C^p$ stratification $(\mathcal{X},\mathcal{Y})$ of $f$ such that $\mathcal{X}$ is compatible with $\mathcal{A}$ and $\mathcal{Y}$ is compatible with $\mathcal{B}$.
\end{thm}

% \begin{ex}
% One can check directly that every Whitney stratification has the properties listed above.
% In particular, the fourth property comes from constructing a controlled vector field that flows one point to another.
% \end{ex}

As observed in the proof of Lemma~\ref{lem:constructible-ECT} and Remark~\ref{rmk:omin-pers}, both the ECT and PHT can be viewed as auxiliary definable constructions associated to an o-minimal set $M$.
Specifically, the ECT is gotten by the pushforward of the indicator function along the map
\[
	\pi: X_M \to S^{d-1}\times \R
\]
and the PHT is associated to the Leray sheaves (the cohomology sheaves of the derived pushforward of the constant sheaf) of this map.
% \color{green}
% Seems to me that re-defining what is meant by stratifications and stratifiable maps and following Loi's presentation~\cite{o-min-strat} would be easiest, pulling out Theorem 1.3 and Theorem 2.2 would be best.
% \color{blue} Done 23:07 Aug 13.
% \color{black}
% Since Theorem~\ref{thm:compatible-strat} says that every o-minimal/definable set can be given a definable $C^p$ Whitney stratification that is compatible with any finite collection of definable subsets containing $M$, then these auxiliary constructions are stratifiable as well.
Since Theorem~\ref{thm:cts-definable-strat} assures us that this map is stratifiable, we get induced stratifications of the codomain of $\pi$ as well as the sphere $S^{d-1}$.
This is the content of the next lemma.

% \color{red}
% Please look at this lemma. I tried to incorporate the comment 15. It needs a further rewrite - see comment 16.
% \color{blue} Done.
% \color{black}
% \color{blue} JMC June 6, 2021. This may be slightly redudant in view of Lemma on Finitely Many changes\color{black}

\begin{lem}\label{lem:strat}
For a general o-minimal set $M\in \Omin_d$, the PHT or the ECT will induce a stratification of $S^{d-1}\times \R$ as well as a stratification of $S^{d-1}$.
Moreover, in the induced stratification of the sphere $S^{d-1}$, a necessary condition for two directions $v$ and $w$ to be in the same stratum is that there is a stratum-preserving homeomorphism of the real line that induces an order preserving bijection between
\[
	\bigcup_{k=0}^{d-1} \bigcup_{(I,j)\in PH_k(M,h_v)} \{b_j^k(I)\} \cup \{d_j^k(I)\} \to  \bigcup_{k=0}^{d-1} \bigcup_{(I',j')\in PH_k(M,h_w)} \{b_{j'}^k(I')\} \cup \{d_{j'}^k(I')\}.
\]
These two sets being the union of all the birth times and death times all the points $(I,j)$ and $(I',j')$ in each of the corresponding persistence diagrams in all degrees $k$, associated to filtering by $h_v$ and $h_w$.
\end{lem}
\begin{proof}
As already proved in Lemma \ref{lem:constructible-ECT}
\[
X_M:=\{(x,v,t)\in M\times S^{d-1} \times \R \mid x\cdot v\leq t\}
\]
is an element of $\Omin_{2d+1}$, and thus admits a 
%decomposition into strata satisfying the conditions listed in Definition~\ref{defn:stratified}.
definable $C^p$ Whitney stratification compatible with $X_M$ and $X_M^c$.
The projection map $\pi:\R^d\times \R^d \times \R \to \R^d \times \R$ onto the last two factors is a continuous definable map.
By Theorem~\ref{thm:cts-definable-strat} the map $\pi$ admits a $C^p$ stratification $(\mathcal{X},\mathcal{Y})$ that is compatible with $X_M$ and the trivializing partition $\{A_i\}$ of $S^{d-1}\times \R$ specified in the proof of Lemma \ref{lem:constructible-ECT}.
%Any such stratification of $S^{d-1}$ we will say is ``induced'' by the ECT or PHT.
 % and hence restricts to a \emph{stratifiable} map $\pi_{S^{d-1}\times \R}: X_M \to S^{d-1}\times \R$.
%Note that the fiber of this map over a direction $(v,t)\in S^{d-1}\times \R$ is the space $\{x\in M\mid x\cdot v\leq t\}$.
%A necessary condition for two points $(v,t)$ and $(w,s)$ to be in the same stratum is that the spaces $\{x\in M\mid x\cdot v\leq t\}$ and $\{x\in M\mid x\cdot w\leq s\}$ be definably homeomorphic and in particular they have the same homology and Euler characteristic.

As was done in Lemma \ref{lem:constructible-ECT}, one can also project further from $S^{d-1}\times \R$ onto $S^{d-1}$ and this map will be a continuous definable map, which in turn will admit a definable $C^p$ stratification that is compatible with the trivializing partition $\{A_i\}$ of $S^{d-1}\times \R$ and $\{B_i\}$ of $S^{d-1}$.
% The fibers of the map $S^{d-1}\times \R \to S^{d-1}$ over any pair of directions $v,w\in S^{d-1}$ in the same stratum will be stratified lines, stratified in a way that is compatible with the $\{F_{ij}\}$ in Lemma \ref{lem:constructible-ECT}. 
Lemma \ref{lem:constructible-ECT} provides definable homeomorphisms $\pi^{-1}(B_i)\cap A_j \to B_i\times F_{ij}$, which can be pieced together to find a definable homeomorphism $\pi^{-1}(B_i) \to B_i\times \R$
where $\R$ is stratified in a way that is compatible with the the partition $F_{i1},\ldots,F_{in}$ of $\R$.
If $v$ and $w$ are in the same stratum of $B_i$, then restricting this homeomorphism to the pre-images of both will induce a zig-zag
\[
\pi^{-1}(v)\hookrightarrow \pi^{-1}(B_i) \cong B_i\times \R \cong \pi^{-1}(B_i) \hookleftarrow \pi^{-1}(w)
\]
that specializes to an order and stratum-preserving homeomorphism from $\pi^{-1}(v)$ to $\pi^{-1}(w)$.
As a reminder, every topological change in the filtration induced by $h_v$ is witnessed by a $0$-cell in $\{F_{ij}\}$.
Since each finite birth or death time in the persistence diagram $PH_k(M,h_v)$ corresponds to a homological change, and thus a topological change, the homeomorphism $\pi^{-1}(v)\cong \pi^{-1}(w)$ will match $0$-cells of $\pi^{-1}(v)$, which include births/deaths of inessential features, with $0$-cells in $\pi^{-1}(w)$. This completes the proof.
% To see this, consider any definable path $\gamma$ that is a homeomorphism onto its image and which connects $v=\gamma(\epsilon)$ and $w=\gamma(1-\epsilon)$ while contained in a single stratum of $S^{d-1}$.
% We can now refine our stratification of $\pi_{S^{d-1}\times \R}$ to be further compatible with the definable set $\gamma\left((0,1)\right)$.
% Note that the strata mapping to $\gamma\left((0,1)\right)$ can only be $1$ and $2$ manifolds because any $0$-dimensional stratum would fail the submersion condition.
% Following the $1$-manifolds over $v=\gamma(\epsilon)$ to $w=\gamma(1-\epsilon)$ establishes the definable homeomorphism claimed in the statement. 
%The fiber over this map in direction $v$ is $$\{(x,t) \in \R^{d+1} \mid (x,t) \in M_{v,t}\}.$$
% The fibers of this further projection will be stratified real lines, which are (possibly) refinements of the stratification given by taking all the critical values of $h_v$ and $h_w$.
% \color{blue} Need to clean up these last few bits of the argument.
% \color{black}
% Consequently, since stratifications imply local triviality of the stratified structure of the real line induced by the birth and death times of the persistence diagrams in each degree $k$ since each of these imply some topological change in the fiber of $\pi_{S^{d-1}\times \R}$.
\end{proof}

Lemma \ref{lem:strat} simply guarantees the existence of a stratification of the sphere and gives criteria for determining when two directions are in the same stratum.
To provide more explicit relationships between Euler curves or persistence diagrams associated to vectors in the same stratum, we specialize to sets $M\in \Omin_d$ that are (PL embedded) geometric simplicial complexes.
We now recall some of the basic definitions.

\begin{defn}
A \define{geometric $k$-simplex} is the convex hull of $k+1$ affinely independent points $v_0,v_1, \ldots v_k$ and is denoted $[v_0,v_1,\ldots,v_k]$.
We call $[u_0, u_1, \ldots u_j]$ a \define{face} of $[v_0,v_1, \ldots v_k]$ if $\{u_0, u_1, \ldots u_j\}\subset \{v_0,v_1, \ldots v_k\}$.
\end{defn} 

\begin{ex}
For example, consider three points $\{v_0,v_1,v_2\}\subset \R^d$ that determine a unique plane containing them. 
The $0$-simplex $[v_0]$ is the vertex $v_0$. 
The $1$-simplex $[v_0,v_1]$ is the edge between the vertices $v_0$ and $v_1$.
Note that $v_0$ and $v_1$ are faces of $[v_0,v_1]$.
The $2$-simplex $[v_0, v_1, v_2]$ is the triangle bordered by the edges $[v_0,v_1]$, $[v_1, v_2]$ and $[v_0, v_2]$, which are also faces of $[v_0, v_1, v_2]$.
\end{ex}

\begin{rmk-defn}[Orientations]
Traditionally, we view the order of vertices in a $k$-simplex as indicating an equivalence class of orientations, where if $\tau$ is a permutation then $[v_0,v_1,\ldots,v_k] = (-1)^{\operatorname{sgn}(\tau)}[v_{\tau(0)}, v_{\tau(1)}, \ldots , v_{\tau(k)}]$. 
However, for the purposes of this paper we can ignore orientation. 
\end{rmk-defn}

\begin{defn}
A \define{finite geometric simplicial complex} $K$ is a finite set of geometric simplices such that
\begin{itemize}
\item[(1)] Every face of a simplex in $K$ is also in $K$;
\item[(2)] If two simplices $\sigma_1,\sigma_2$ are in $K$ then their intersection is either empty or a face of both $\sigma_1$ and $\sigma_2$.
\end{itemize}
\end{defn}

\begin{rmk-defn}[Re-Triangulation]
Strictly speaking we only care about the embedded image of the finite simplicial complex. 
We consider different triangulations as equivalent; our uniqueness results will always be up to re-triangulation.
\end{rmk-defn}

%Let $K\subset \R^d$ be a finite simplicial complex with vertex set $X$. 

For a pair of distinct points $x_i,x_j\in \R^d$ let $W(\{x_i, x_j\})\subset \R^d$ be the hyperplane $\{v\in \R^d: x_i\cdot v= x_j \cdot v\}$ that is orthogonal to the vector $x_i-x_j$. 
This hyperplane divides the sphere $S^{d-1}$ into two halves depending on whether $h_v(x_i)>h_v(x_j)$ or $h_v(x_j)>h_v(x_i)$. 
More generally, for a set of points $X=\{x_1, x_2, \ldots x_k\}$ we can define 
\begin{align*}
W(X):=\bigcup_{i\neq j} W(\{x_i, x_j\})
\end{align*}
as the union of the hyperplanes determined by each pair of distinct points.

\begin{defn}
Let $X=\{x_1, x_2, \ldots x_k\}$ be a finite set of points in $\R^d$ and let $W(X)$ be the union of hyperplanes determined by each pair of points.
The hyperplane union $W(X)$ induces a stratification of $S^{d-1}$, which we call the \define{hyperplane division} of $S^{d-1}$ by $X$.
Define  $\Sigma(X)$ as  $S^{d-1} \setminus W(X)$.
\end{defn}

% \color{red}
% Dramatically changed the remark in respone to comment 19. Any comments? Do we need to prove?
% \color{black}
\begin{rmk-defn}[Stratified space structure]
For a direction $v$, consider the simplicial complex $P(v)$ over the vertex set $X$ where $\Delta \in P(v)$ if $v$ is perpendicular to the affine plane spanned by the $x_i\in \Delta$. 
We can define a function $N:S^{d-1}\to \Z$ by $$N(v)=\sum_{\text{locally maximal } \Delta\subset P(v)} \dim(\Delta).$$  If the points in $X$ lie in generic position then the function $N$ induces a stratified space structure on the sphere where the $d-k-1$ dimensional strata correspond to the connected components of $N^{-1}(k)$. Notably the top dimensional strata (with dimension $d-1$) are the connected components on $\Sigma(X)$. 
Notice that even when the hyperplanes are not in general position we can find, by virtue of Theorem~\ref{thm:compatible-strat}, a definable $C^p$ Whitney stratification of $S^{d-1}$ that is compatible with the hyperplanes and all of their intersections.
\end{rmk-defn}

\begin{lem}
Let $X=\{x_1, x_2, \ldots x_k\}\subset \R^d$.
If $v_1,v_2\in S^{d-1}$ are in the same stratum of $\Sigma(X)$, 
then the order of $\{v_1\cdot x_i\}_{i=1}^k$ is the same as the order of the $\{v_2 \cdot x_i\}_{i=1}^k$.
\end{lem} 
\begin{proof}	
Observe that $v_1$ and $v_2$ lie in the same hemisphere of $W(\{x_i,x_j\})$ and so $h_{v_1}(x_i)>h_{v_1}(x_j)$ if and only if $h_{v_2}(x_i)>h_{v_2}(x_j)$.
\end{proof}

\begin{defn}
For each vertex $x\in K$,  the \define{star} of $x$, denote $\St(x)$ is the set of simplices containing $x$.
Given a function $f:X\to \R$ we can define the \define{lower star} of $x$ with respect to $f$, denoted $\LwSt(x,f)$, as the subset of simplices $\St (x)$ whose vertices have function values smaller than or equal to $f(x)$. Both stars and lower stars are generally not simplicial complexes as they are not closed under the face relation.
\end{defn}

\begin{rmk-defn}[Topological Interpretation of the Star]
Although a geometric simplex is defined here in terms of convex hulls, which are closed when viewed as topological spaces, they are better viewed as topological spaces via their interiors.
For example, if $K$ is the simplicial complex consisting of the 1-simplex $[x,y]$ along with its two faces $[x]$ and $[y]$, then according to the above definition $\St(x)=\{[x], [x,y]\}$.
If we identify each geometric simplex with its interior, then the star is the ``open star,'' which in this case is the half-open interval $\St(x)=[x,y)$.
Below we will give a combinatorial formula for computing the Euler characteristic of the star, which in this case is $\chi(\St(x))=1-1=0$.
Note that if one is used to thinking in terms of Euler characteristic of the underlying space, then one must use compactly-supported Euler characteristic of the open star in order for these viewpoints to cohere.
\end{rmk-defn}

For a finite geometric simplicial complex $K\subset \R^d$ with vertex set $X$, the height function $h_v:K\to\R$ is the piece-wise linear extension of the restriction of $h_v$ to the set of vertices $X$. 
When $v\notin W(X)$ then all the function values of $h_v$ over $X$ are unique. This implies that each simplex belongs to a unique lower star, namely to the vertex with the highest function value. 
We will make essential use of the following result.

% \color{red}
% We need a citation that the  sublevel set $f^{-1}(-\infty, t]$ is homotopic to $\bigcup_{\{x\in X:v\cdot x\leq t\}}\text{LwSt}(x,f)$ 
% \color{black}

\begin{prop}[Lem.~2.3 of~\cite{BB-Morse} and \S VI.3 of~\cite{EH-book}]\label{prop:lower-star}
For a piecewise linear function $f: K \to \R$ defined on a finite geometric simplicial complex $K$ with vertex set $X$, we have that for every real number $t\in\R$, the sublevel set $f^{-1}(-\infty, t]$ is homotopic to $\bigcup_{\{x\in X:v\cdot x\leq t\}}\LwSt(x,f)$.
\end{prop}

% Note that the sublevel set $f^{-1}(-\infty, t]$ is homotopic to $\bigcup_{\{x\in X:v\cdot x\leq t\}}\text{LwSt}(x,f)$, this holds due to a deformation retraction; see Section VI.3 of~\cite{EH-book} for a proof. 
Since both $\bigcup_{\{x\in X:v\cdot x\leq t\}}\text{LwSt}(x,f)$ and $f^{-1}(-\infty,t]$ are also compact we conclude that they have the same definable Euler characteristic or Euler characteristic with compact support. 
We will use the notation
\[
K^{(x,v)}=\left\{\Delta\in \St(x) \subseteq K \mid x\cdot v=\max_{y\in \Delta}\{ y\cdot v\} \right\}
\]
to denote the lower star of $x$ in the filtration by the height function in direction $v$. 

%\color{blue} Justin says: This Lemma seems wrong as stated. Consider a 2-simplex and filter in a generic direction.
%\color{black}

\begin{lem}\label{lem:KxU}
Let $K\subset \R^d$ be a finite simplicial complex with vertex set $X$. 
Let $U$ be a connected subset of $\Sigma(X)$.
Then for fixed $x\in X$, the lower stars $K^{(x,v)}$ are all the same for all $v\in U$. 
We will sometimes denote the lower star by $K^{(x,U)}$ to highlight this consistency.
Moreover, for all $v\in U$ we have the following formula for the Euler Characteristic Transform:
$$\ECT(K, v)=\sum_{x \in X} 1_{\{[v\cdot x,\infty)\}} \chi(K^{(x,U)})$$
\end{lem}

\begin{proof}	
We can see that the $K^{(x,v)}$ agree for all for all $v\in U$ by the 
observation that the vertices of $X$ appear in the same order in each of the $h_v$.
This $K^{(x,v)}$ is the subset of cells that are added at the height value $h_v(x)$.

Note that the change in the Euler characteristic of the sublevel sets of $h_v$ as the height value passes $h_v(x)$ is 
$$\chi(K^{(x,U)})=\sum_{\Delta \in K^{(x,U)}} (-1)^{\dim \Delta}.$$ 
Here we use the notation $\chi(K^{(x,U)})$ for the (compactly-supported) Euler characteristic of this set of simplices.
Consequently, we can write the Euler characteristic curve for $K$ in direction $v\in U$ as
$$\ECT(K, v)=\sum_{x\in X} 1_{\{[v\cdot x,\infty)\}} \chi(K^{(x,U)})$$
Note that many of the elements in this sum are zero.
\end{proof}

\begin{ex}
As a simple example, let $K$ be an embedded geometric 2-simplex $[x,y,z]$ along with all of its faces and suppose $v$ is a direction in which, filtering by projection to $v$, the vertices appear in the order $x$, then $y$, then $z$.
Note that $K^{(x,v)}=[x]$, which has Euler characteristic 1; $K^{(y,v)}=[x,y] \cup [y]$, which has Euler characteristic $(-1)+(-1)^0=0$; and $K^{(z,v)}=[x,y,z] \cup [y,z] \cup [x,z] \cup [z]$, which has Euler characteristic $(-1)^2+(-1)+(-1)+(-1)^0=0$.
\end{ex}

% We can observe that these Euler characteristic curves over $U$ are essentially linear in $v$. 
% It is just that the linear interpolation occurs inside the dot products within the summand and not on the function level.

We now show that when two vectors $v$ and $w$ lie in the same connected component of $\Sigma(X)$, then the Euler curve or the persistence diagrams for $h_v$ can be used to determine the Euler curve or persistence diagrams for $h_w$.
We thank an anonymous referee for suggesting the following improved statement and proof.

% \color{red}
% Below is the combined streamlined proof as recommended in comment 23. Please look over
% \color{black}

\begin{prop}\label{prop:deduce}
Let $K\subset \R^d$ be a finite simplicial complex with known vertex set $X$. 
If $v,w\in S^{d-1}$ lie in the same connected component of $\Sigma(X)$, 
then given the Euler curve for $v$ we can deduce the Euler curve for $w$.
Similarly, given the $k$th persistence diagram associated to filtering by $h_v$ we can deduce the $k$-th persistence diagram for filtering by $h_w$.
\end{prop}
\begin{proof}

%Both the results can have the same proof. Readers shouldn't have to think about the details of pairing simplices and killing classes and so on. The truth of these results resides at simpler level.
% It can be helpful to consider a functions on $K$ that are related to the height functions but is constant on each (open) simplex.

% Associated to the height function $h_v:K \to \R$ is the height function over $K\subset \R^d$ considered as a constructible subset of $\R^d$. 
% In the piecewise linear setting $K$ is also a simplicial complex with vertex set $X$. 
% We will construct an auxiliary function $k_v$ that has the same sublevel set persistent homology as $h_v$. 
Consider the height function $h_v:K \to \R$ on our PL embedded simplicial complex $K\subset \R^d$.
Recall that the vertex set of $K$ is $X$.
Associated to the function $h_v$ is another (typically discontinuous) function $k_v: K \to \R$ defined by setting $k_v(x) = h_v(\hat{x})$ where $\hat{x}$ is the unique vertex whose lower star contains the simplex in whose interior $x$ belongs.
Notice that this has the effect of increasing the value on the interior of a simplex to the maximum value obtained on any of its vertices.
Consequently, for all $x\in K$ we have that $h_v(x)\leq k_v(x)$.
This implies that for every $t\in \R$ we have the containment $k_v^{-1}(-\infty, t] \subseteq h_v^{-1}(-\infty, t]$.
Notice that the sublevel set $k_v^{-1}(-\infty, t]=\bigcup_{\{x\in X:v\cdot x\leq t\}}\text{LwSt}(x,h_v)$.
By virtue of Proposition~\ref{prop:lower-star}, we have that this inclusion of sublevel sets is a homotopy equivalence for every $t$, so $h_v$ and $k_v$ have identical persistence diagrams in all dimensions.

% Let $k_v(x) = h_v(\hat{x})$ where $\hat{x}$ is the unique vertex whose lower star contains the simplex in whose interior $x$ belongs. 
% Note that $k_v \geq h_v.$ This implies that the sublevelset filtration of $(K, k_v)$ includes into the sublevelset filtration of $(K, h_v)$. Furthermore, this inclusion is always a homotopy equivalence by 
% \color{red}

% Moreover, the result at the top of page 14 is precisely that this inclusion is a homotopy equivalence at each level of the filtration. 

% \color{black}

% This implies that $(K, h_v)$ and $(K, k_v)$ have identical persistence diagrams in each homological dimension.

The virtue of defining the auxiliary function $k_v$ is that it easy to compare with $k_w$ when $v$ and $w$ belong to the same component of $\Sigma(X)$.
This is because the level sets of $k_v, k_w$ are empty except at the values of the inner product $x\cdot v$ and $x\cdot w$. 
At those values, the level sets are precisely the lower stars.
As such, let $\phi : \R \to \R$ be any order-preserving homeomorphism that maps the values $x\cdot v$ to the values $x\cdot w$ for all $x\in X$. 
Such a $\phi$ exists because the order of the $\{x\cdot v\}$ and $\{x \cdot w\}$ for $x\in X$ are the same.
% The level sets of $k_v, k_w$ are empty except at the values $(x\cdot v)$ and $(x\cdot w).$ At those values, the level sets are precisely the lower stars. 
The mapping $\phi$ provides a bijection between these sets of values, with the property that the corresponding levelsets (i.e. lower stars) are equal, by virtue of Lemma~\ref{lem:KxU}.
This implies that $k_w = \phi \circ k_v.$
It immediately follows that $\ECT(K,v) = \ECT(K,w) \circ \phi^{-1}$. 
Similarly, $\PHT(K,w)$ can be obtained from $\PHT(K,v)$ by transforming the persistence diagram plane using the function $(\phi,\phi):\R^{2+} \to \R^{2+}$.
This is because the birth time $b_i$ and death time $d_i$ of a feature in $\PHT(K,v)$ is transformed to a birth time $\phi(b_i)$ and death time $\phi(d_i)$ for a feature in $\PHT(K,w)$.

% Observe that we are appealing to the trivial fact that the persistence diagram transforms in the obvious way when we reparametrize the real line.
\end{proof}

\section{Uniqueness of the Distributions of Euler curves up to $O(d)$ actions}\label{sec:affine}

%\color{red}
%We need to introduce the pushforward measure somewhere in this intro to the section.
%\color{blue}
%added a sentence to the second paragraph, but I am a little skeptical about how important this concept is: all we use the pushforward measure to do is to show that the images are the same.
%linearity and orthogonality are determined by the fact that $K_1$ and $K_2$ are PL.
%\color{black}

In order to guarantee that two ``close'' shapes have close transforms---using either the ECT or PHT---one must first align or register the shapes' orientations in space.
For example, if one wanted to compare simplicial versions of a lion and a tiger, we would need to first embed them in such a way that they are facing the same direction. 
To rephrase this, for centered shapes we would like to make them as close as possible by using actions of the special orthogonal group $SO(d)$. 
If we wished to optimize their alignment by also allowing reflections, we would like to make them as close as possible by using actions of the orthogonal group $O(d)$.
In general, aligning or registering shapes is a challenging problem~\cite{Puente}.

In this section we show how studying distributions on the space on Euler curves---or persistence diagrams, since homology determines Euler characteristic---can bypass this process. We can construct distributions of topological summaries through the pushforward under the topological transform of the uniform measure over the sphere. If $\mu$ is the Lesbesgue measure over $S^{d-1}$ then the pushforward measure $\ECT_*(\mu)$ is a measure over the space of Euler curves defined by $\ECT_*(\mu)(A)=\mu(\ECT^{-1}(A))$. These pushforward measures are naturally invariant of $O(d)$ actions as the Lebesgue measure on the sphere is $O(d)$ invariant; acting on a shape and acting on the space of directions using the same element of $O(d)$ produces the same Euler curve.
The key development of this section, Theorem~\ref{thm:align}, is the proof that for ``generic'' shapes the pushforward of the Lebesgue measure to the space of Euler curves uniquely determines that shape up to an $O(d)$ action. 
Additionally, our proof of this theorem is constructive. 
If we know the actual transforms of two generic shapes and we recognize they produce the same distribution of Euler curves, then we construct an element of $O(d)$ that relates them.
However, the deeper implication is that knowing the distribution of Euler curves for a generic shape is a sufficient statistic for shape comparison.
Continuing the aforementioned example, we can compare an arbitrarily embedded tiger and lion without ever aligning them. The theoretical ideas stated in this section were translated into an algorithm that ``aligns" shapes without requiring correspondences in \cite{Wang}. 

We now specify what we mean by a ``generic shape.''

\begin{defn}
A geometric simplicial complex $K$ in $\R^d$ with vertex set $X$ is \define{generic} if
\begin{itemize}
\item[(1)] the Euler curves for the height functions $h_v$ are distinct for all $v\in S^{d-1}$, and
\item[(2)] the vertex set $X$ is in general position.
\end{itemize}
\end{defn}

Our first task is to bound from below the number of ``critical values'' on a generic geometric complex when filtered in a generic direction.
Since critical points and critical values are usually understood in a differentiable setting, and the shapes are dealing with are only piece-wise differentiable, we make precise what we mean by this intuitive term. 

% \color{red}
% Changed the definitions here and the following remark in response to comment 27. Please look over.
% \color{black}

\begin{defn}\label{defn:critical-value}
Let $f: K \to \R$ be a continuous function and $\chi_{f}: \R \to \Z$ the Euler curve of $f$.
We say that $t$ is an \define{Euler regular value} of $f: K \to\R$ if there is some open interval $I\subseteq \R$ containing $t$ where $\chi_f$ is constant. 
An \define{Euler critical value} is a real number which is not an Euler regular value.

A real number $t$ is called a \define{homological regular value} if there is some open interval $I\subseteq \R$ containing $t$ such that for all $a<b$ in $I$ the induced maps on homology by inclusion ($i_*:H_*(f^{-1}(-\infty, a]) \to H_*(f^{-1}(-\infty,b])$) are all isomorphisms.
A \define{homological critical value} is any real number that is not a homological regular value.
For the purposes of this paper we will just use the term \define{critical value} when there is no chance for confusion.
\end{defn}

\begin{rmk-defn}%[Homological Critical Value]
There is some historical variation of the definition of a homological critical value and we refer the reader to~\cite{CritValue} for a comparison and contrast of two candidate definitions along with several interesting examples.
\end{rmk-defn}

We now state a lower bound on the number of critical values of a generic complex when filtered in a direction in $\Sigma(X)$.

\begin{lem}\label{lem:d-1}
If $K$ is a generic finite geometric simplicial complex embedded in $\R^d$ with vertex set $X$
then for all $v\in \Sigma(X)$ the height function $h_v$ has at least $d-1$ Euler critical values, i.e.~values for which the Euler characteristic of the sub-level set changes.

Furthermore, for each connected component $U\subset \Sigma(X)$ there exists $d-1$ linearly independent vertices $\{x_1^U, x_2^U, \ldots, x_{d-1}^U\}$ which are Euler critical for $ECT(K,v)$ for all $v\in U$.
\end{lem}

\begin{proof}
%Since $W(X)\cap S^{d-1}$ is of measure zero in $S^{d-1}$ it is sufficient to show that for all $v\in \Sigma(X)$, the height function has at least $d-1$ critical values. 
Let $U$ be a connected component of $\Sigma(X)$. The formula in Lemma~\ref{lem:KxU} implies the number of Euler critical values for $h_v$ ($v\in U$) is the number of vertices $x$ such that $\chi(K^{(x,U)})\neq 0$. Suppose that there are $k$ critical vertices  $x_1, x_2, \ldots x_k$, so that for all $v\in U$ 
$\ECT(K, v)=\sum_{i=1}^k1_{\{[v\cdot x_i,\infty)\}} \, \chi(K^{(x_i,U)}).$

%Suppose by way of contradiction that this number of critical values is fewer than $d-1$. This implies that there exists an $x_1, x_2, \ldots x_k$ (with $k\leq d-2$) such that for all $v\in U$ 
%$\ECT(K, v)=\sum_{i=1}^k1_{\{[v\cdot x_i,\infty)\}} \, \chi(K^{(x_i,U)}).$

Define a function $g:U\to \R^k$, with $g(v)=(v\cdot x_1, v\cdot x_2, \ldots,  v\cdot x_k)$ which is clearly continuous and is injective by our genericity assumptions. As $U$ is an open subset of a $S^{d-1}$ there is a open subset $A\subset \R^{d-1}$ and a homeomorphism $f:A\to U$. Their composition, $g\circ f:A\to \R^k$, is continuous and injective.  By Brouwer's Invariance of Domain Theorem we know that $A$ is homeomorphic to $g(f(A))$ and thus $\dim(g(f(A)))=d-1$. Since $f(U)\subset \R^k$ we conclude that $k\geq d-1$. 

By restricting to the first $d-1$ critical vertices if needed, we have found vertices $\{x_1^U, x_2^U, \ldots, x_{d-1}^U\}$ which are Euler critical for $ECT(K,v)$ for all $v\in U$. We know that these $x_i^U$ are linearly independent by our genericity assumptions.
%There must be some $w,v\in U$ with $g(v)=g(w)$. This implies $\ECT(K, w)=\ECT(K, v)$ contradicting our genericity assumption the Euler curves are all distinct. 

\end{proof}

%\color{red}
%MAYBE MOVE THIS TO CONTINUITY DISCUSSIONS?
%% DONE BY JMC
% We offer the following remark on the Euler Characteristic Transform, particularly its continuity and its image.

%\color{black}

Before proceeding to our main result, we need two technical lemmas.

\begin{lem}\label{lem:orthogonal}
Let $\ell:\R^d\to\R^d$ be a linear map. If there is a non-empty open subset $U$ of $S^{d-1}$ such that $\| \ell(v) \|=1$ for all $v\in U$, then $\ell$ is an orthogonal transformation.
\end{lem}

\begin{proof}
Without loss of generality we can assume that $U$ is an open ball in $S^{d-1}$. We want to show that for $v,w\in U$ that $\langle \ell(v),\ell(w) \rangle=\langle v,w\rangle$. Since $v,w\in U$, with $U$ an open ball in $S^{d-1}$, we know that $(v+w)/\|v+w\|\in U$ and hence
$$\|\ell(v+w)\|=\left\| \|v+w\| \ell\left( \frac{v+w}{\|v+w\|}\right) \right\|=\|v+w\|.$$

We thus can compute the inner product of $\ell(v)$ and $\ell(w)$ by

\begin{align*}
\langle \ell(v),\ell(w)\rangle &=\frac{1}{2}(\langle \ell(v+w), \ell(v+w)\rangle -\langle \ell(v),\ell(v)\rangle -\langle \ell(w), \ell(w)\rangle)\\
&=\frac{1}{2}\left(\|\ell(v+w)\|^2 - \| \ell(v)\|^2 -\| \ell(w)\|^2\right)\\
&=\frac{1}{2}\left(\|v+w\|^2 - \| v\|^2 -\| w\|^2\right)\\
&=\langle v, w \rangle
\end{align*}

%Since $\alpha:\R^d\to\R^d$ is linear and preserves norms over $U$, we know that $\alpha$ must preserve norms over all rescalings of $U$. This implies that $\alpha$ preserves norms over a non-empty open subset of $\R^d$. Choose an open $B(x,r)\in \R^d$ inside the set of rescalings of $U$. 
We can generate all of $\R^d$ from linear combinations of points in $U$. The inner product is bi-linear and so $\ell$ preserves the inner product over all of $\R^d$.
As $\ell$ preserves the inner product over $\R^d$, $\ell$ must be an orthogonal transformation. 
\end{proof}

\begin{lem}\label{lem:connected}
Let $\mathcal{S}$ be a definable stratification of $S^{d-1}$. 
Let $\mathcal{S}^{k}$ be the set that indexes connected components of $k$-dimensional strata within $\mathcal{S}$.
%let $E:=\mathcal{S}^{d-2}$ be the set of $(d-2)$ strata of $\mathcal{S}$, and let $\rho:V\cup E \to \Gamma$ be the map that realises each element of $V$ or $E$ as its corresponding subset in $S^{d-1}$. 
Let $G$ be the graph with vertex set that indexes top-dimensional strata, i.e. $V(G):=\mathcal{S}^{d-1}$, and with edge set that indexes codimension-1 strata, i.e. $E(G):=\mathcal{S}^{d-2}$.
We say that $i$ and $j$ are connected by edge $k$ if $S_k\leq S_i$ and $S_k\leq S_j$, i.e. $S_k\subset \overline{S_i}$ and $S_k\subseteq \overline{S_j}$.
With this construction, the graph $G$ is connected.
\end{lem}

\begin{proof}
Let $X_{d-3}$ be the union of strata of dimension $(d-3)$ and below in a definable stratification of $S^{d-1}$.
Let $Y:=S^{d-1}\setminus X_{d-3}$ be the complement of this union, which carries a definable stratification by restriction of the stratification $\mathcal{S}$ of $S^{d-1}$.
By Alexander Duality we have that $\tilde{H}_0(Y)\cong \tilde{H}^{d-2}(X_{d-3})=0$, which implies that $Y$ is connected.
This proves that the union of top-dimensional strata and codimension-1 strata is a connected subspace of $S^{d-1}$.

Now consider any partition of the vertex set $V(G)$ into subsets $V_0$ and $V_1$.
Let $Y_0$ and $Y_1$ denote the union of the corresponding top-dimensional strata represented by vertices in $V_0$ and $V_1$, respectively. 
Since the strata partition $Y$ we have that $Y=\overline{Y_0}\cup \overline{Y_1}$ is expressible as a union of two closed (in $Y$) sets.
Moreover, we know from the Axiom of the Frontier that $\overline{Y_0}\cap \overline{Y_1}$ is a union of $(d-2)$-strata.
If this intersection is empty, then there are no $(d-2)$-strata and we have that $Y$ is the union of a single top-dimensional stratum, thus proving that either $V_0$ or $V_1$ must be empty.
If the intersection $\overline{Y_0}\cap \overline{Y_1}$ is non-empty, then there must be some $(d-2)$-dimensional stratum in this intersection and thus some vertex in $V_0$ must be connected by an edge with some vertex in $V_1$. This completes the proof.
\end{proof}

The following is the main result of this section and is our formal statement about the uniqueness of a distribution over diagrams of curves up to an action by $O(d)$.

\begin{thm}\label{thm:align}
Let $K$ and $L$ be generic geometric simplicial complexes in $\R^d$. 
Let $\mu$ be Lesbesgue measure on $S^{d-1}$. 
If $\ECT(K)_*(\mu)=\ECT(L)_*(\mu)$ (that is the pushforward of the measures are the same), then there is some $\phi \in O(d)$ such that $L = \phi(K)$, that is to say that $L$ is some combination of rotations and reflections of $K$. 
\end{thm}

\begin{proof}
%First we describe the proof at a high-level.
% \color{blue} Do we actually know the images are the same? Seems like the agreement of measures doesn't quite do it.
% Haven't finished smoothing the exposition of the proof below.
% \color{black}
Recall that $\ECT(K)$ and $\ECT(L)$ are continuous and injective maps to the space of Euler curves. 
The hypothesis $\ECT(K)_*(\mu)=\ECT(L)_*(\mu)$ implies that the support of the pushforward measures are the same. 
We now argue that equality of supports implies equality of images.
Firstly, we note that the support of $\ECT(K)_*(\mu)$ (respectively $\ECT(L)_*(\mu)$) contains the image of $\ECT(K)$  (respectively $\ECT(L)$).
To see this, consider an Euler curve in the image and consider an open ball around that curve. Continuity implies that the pre-image is open and thus the Lebesgue measure is positive.
Secondly, we show that any Euler curve not in the image of $\ECT(K)$ is not in the support of $\ECT(K)_*(\mu)$.
% Additionally, any Euler curve not in the image of $\ECT(K_i)$ is not in the support because the map $\ECT(K_i):S^{d-1} \to \CF(\R)$ is closed. 
To see this, we note that Remark~\ref{rmk:Hausdorff-ECs} implies that
$\ECT(K):S^{d-1} \to \CF(\R)$ is a continuous map from a compact space to a Hausdorff space, so the image is relatively closed in the subspace of possible Euler curves.
If an Euler curve is not in the image of $\ECT(K)$, then there is an open set containing it that is disjoint from the image, which has measure zero according to the pushforward measure. Hence the support is a subset of the image. 
The above argument shows that $\ECT(K)_*(\mu)=\ECT(L)_*(\mu)$ implies the images of $\ECT(K)$ and $\ECT(L)$ are the same. Our genericity assumptions imply that $\ECT(L)$ is injective and thus we can define a bijection 
$$\phi= \ECT(L)^{-1} \circ \ECT(K) :S^{d-1}\to S^{d-1}.$$ 
The map $\phi$ is both definable and a homeomorphism. 

To see $\phi$ is a homeomorphism we need to show that $\ECT(L)^{-1}: \im(\ECT(L))\to S^{d-1}$ is continuous which we will do via a contrapositive version of an $\epsilon-\delta$ argument. For $v\in S^{d-1}$ define the function $f_v: S^{d-1}\to [0,\infty)$ by $f_v(w)=d_p(\ECT(L)(v), \ECT(L)(w))$ is continuous. The set $S^{d-1} \backslash B(v, \epsilon)$ is compact and hence $f_v(S^{d-1} \backslash B(v, \epsilon))\subset [0,\infty)$ is a compact interval. This interval cannot contain $0$ because $\ECT(L)$ is injective. This implies that there is some $\delta>0$ with $f_v(w)>\delta$ for all $w\in S^{d-1} \backslash B(v, \epsilon)$. Since $v\in S^{d-1}$ was arbitrary, this means that for all $v\in S^{d-1}$ and all $\epsilon>0$ there exists a $\delta>0$ such that $d_{S^{d-1}}(v,w)>\epsilon$ implies that $d_p(\ECT(L)(v), \ECT(L)(w))>\delta$.
 
We will now show that $\phi$ is definable. 
Let 
$$A=\{(v,t,\ECT(K)(v,t))|(v,t)\in S^{d-1}\times \R\}$$ 
and 
$$B=\{(w,s,\ECT(L)(w,s)):(w,s)\in S^{d-1}\times \R\}.$$ 
Both $A$ and $B$ are definable since the graph of definable functions is definable.
Here we use that $\ECT(K)$ an $\ECT(L)$ are constructible (hence definable) functions on $S^{d-1}\times \R$. 
% Let $Y$ be the set of pairs $(v,w)$ such that for all $t$ there exists a $z$ such that $(v,t,z)\in A$ and $(w,t,z)\in B$. This is a definable set as $A$ and $B$ are definable. By definition this implies $\phi$ and $\phi^{-1}$ are both definable.
Let $P$ be the set of pairs of directions $(v,w)$ such that the Euler curve of $K$ in the direction of $v$, written $\ECT(K)(v)$, is the same as the Euler curve of $L$ in the direction of $w$, written $\ECT(L)(w)$.
Since the pullback of definable sets is definable~\cite[Lemma 11.1.15]{SCA}, we know that $P$ is definable.
$P$ is then the graph of a definable map $\phi:S^{d-1}\to S^{d-1}$.

Let $X$ and $Y$ denote the vertex sets of $K$ and $L$ respectively.
Let $W(X)$ and $W(Y)$ be the hyperplane partitions  of  $X$ and $Y$, respectively,  as defined in the previous section. 
By construction $W(X)$ is the union of $\binom{|X|}{2}$ hyperspheres, i.e.~the intersections of the $d-1$ planes $W(\{x_i,x_j\})$ with $S^{d-1}$. 
We also know that $\phi^{-1}(W(Y))$ is homeomorphic to a union of $\binom{|Y|}{2}$ hyperspheres.
Furthermore $\phi^{-1}$ is definable and thus $W(X)\cap \phi^{-1}(W(Y))$ is definable.  There exists a stratification of $S^{d-1}$ into finitely many strata such that $W(X)\cap \phi^{-1}(W(Y))$ is contained in the lower dimensional strata. We will be looking at the restriction of $\phi$ to these top dimensional strata and showing that these restrictions are each the restriction of elements of $O(d)$, that is restrictions of an orthogonal map to their respective strata as the domain. We will then compare $d-1$ strata separated by a $d-2$ strata and show that their corresponding restrictions of $\phi$ agree as elements of $O(d)$.

Let $U\subset S^{d-1}$ be a connected open set with $U\cap W(X)=\emptyset$ and $\phi(U)\cap W(Y)=\emptyset$, equivalently $U\cap \phi^{-1}(W(Y))=\emptyset$. This implies that $U$ is a connected subset of $\Sigma(X)$ and $\phi(U)$ is a connected subset of $\Sigma(Y)$ which means we can apply Lemma~\ref{lem:KxU} for both $K$ and $L$ separately. 
%Note that here $U$ is not itself an entire stratum of $\Sigma(X_1)$ or of $\phi^{-1}(\Sigma(X_2)))$, but simply some connected subset of their intersection.
%However, as a subset it inherits all the properties proved in the preceding sections.

Let $X_U$ and $Y_{\phi(U)}$ respectively denote the vertices of $K$ or $L$ such that $\chi(K^{(x,U)})\neq 0$ or $\chi(L^{(x,\phi(U))})\neq 0$.  
From Lemma~\ref{lem:KxU} we know that
\[
\ECT(K, v)= \sum_{x\in X_U} 1_{\{[v\cdot x,\infty)\}} \chi(K^{(x,U)})
\]
and also that
\[
\ECT(L, \phi(v))=\sum_{y\in Y_{\phi(U)}}1_{\{[\phi(v)\cdot y,\infty)\}} \chi(L^{(x,\phi(U))}).
\]

Recall that the order of the values $\{v\cdot x\}$ over $x\in X_U$ are the same for all $v\in U$ and that this determines a total ordering over $X_U$. The same is true for the values $\{\phi(v) \cdot y\}$ over $y\in Y_{\phi(U)}$ again determining a total ordering over $Y_{\phi(U)}$. Since $\ECT(K_1, v)=\ECT(K_2, \phi(v))$ for all $v\in U$ we know that $|X_U|=|Y_{\phi(U)}|$ and furthermore we can consistently index the elements of $X_U$ and $Y_{\phi(U)}$ such that 
$$x_i \cdot v=y_i \cdot \phi(v)$$ 
for all $i=1, \ldots , |X_U|$ and for all $v\in U$.

For clarity of exposition we shall first consider the case where $|Y_{\phi(U)}|\geq d$. From our genericity assumptions, $\{y_1, \ldots , y_d\}$ are linearly independent. This means that we can define the matrix 
$$M_U = \begin{bmatrix}
\horzbar & y_{1} &\horzbar \\
\horzbar & y_{2}&\horzbar \\
& \vdots & \\
\horzbar & y_{d} & \horzbar
\end{bmatrix}^{-1}
\begin{bmatrix}
\horzbar & x_{1} &\horzbar \\
\horzbar & x_{2}&\horzbar \\
& \vdots & \\
\horzbar & x_{d} & \horzbar
\end{bmatrix}
$$
%\brows{
%x_{i_1}\\
%x_{i_2}\\
%\rowsvdots\\
%x_{i_d}
%}
which has
$$M_U^T:= 
\begin{bmatrix}
\vertbar & \vertbar & &\vertbar \\
x_{1} & x_{2}& \ldots & x_d \\
\vertbar & \vertbar & &\vertbar
\end{bmatrix}
\begin{bmatrix}
\vertbar & \vertbar & &\vertbar \\
y_{1} & y_{2}& \ldots & y_d \\
\vertbar & \vertbar & &\vertbar
\end{bmatrix}^{-1}
$$
and hence $M_U^Ty_i=x_i$ for all $i$.

We will show that $\phi(v)=M_U v$ for all $v\in U$.% (here $t$ denotes the transform of $A$). 
%By construction $\alpha y_i=x_i$ for all $i=1, \ldots , d$.
For each $v\in U$ we have
$$  y_i \cdot  \phi(v) =  x_i \cdot v = M_U^Ty_i \cdot v =y_i \cdot M_U v.$$
Since this inner product holds for basis $\{y_i\}$ we know that $\phi(v)=M_U v$ for all $v\in U$. 
%This implies that $\phi|_U$ is linear, by which we mean that $\phi|_U$ is the restriction of the linear map from $A_U:\R^d \to \R^d$.

%This ordering provides a bijection 
%\[
%	\psi:X_1(U) \to X_2(\phi(U)) \quad \text{where} \quad v\cdot x=\phi(v)\cdot \psi(x) \quad \forall x\in X_1(U) \quad \text{and} \quad v\in U.
%\]

%If $X_2(\phi(U))$ contains $d$ linearly independent $y_{i_1}, y_{i_2} \ldots y_{i_d}$ then we could directly get linearity as 
%$\phi(v)\cdot y_{i_j}= v\cdot x_{i_j}$ for all $j$ and all $v\in U$ implies that 
%$$\phi|_U = \begin{bmatrix}
%\horzbar & y_{i_1} &\horzbar \\
%\horzbar & y_{i_2}&\horzbar \\
%& \vdots & \\
%\horzbar & y_{i_d} & \horzbar
%\end{bmatrix}^{-1}
%\begin{bmatrix}
%\horzbar & x_{i_1} &\horzbar \\
%\horzbar & x_{i_2}&\horzbar \\
%& \vdots & \\
%\horzbar & x_{i_d} & \horzbar
%\end{bmatrix}
%%\brows{
%%x_{i_1}\\
%%x_{i_2}\\
%%\rowsvdots\\
%%x_{i_d}
%%}
%$$
% and hence $\phi|_U$ is linear, by which we mean that $\phi|_U$ is the restriction of a linear map from $\R^d$ to $\R^d$.
% 
Now suppose that we do not have $|Y_{\phi(U)}|<d$.  We know that $X_U$ and $Y_(\phi(U))$ each must contain at least $d-1$ linearly independent points from Lemma~\ref{lem:d-1}. To be able to use the arguments from the case where $|Y_{\phi(U)}|\geq d$ it will be sufficient to find vectors $w_x$ and $w_y$ such that $w_x$ is linearly independent of $X_U$, $w_y$ is linearly independent of $Y_{\phi(U)}$ and such that $\phi(v) \cdot w_y = v \cdot w_x$ for all $v\in U$. We can then apply the same reasoning merely substituting $x_d$ (respectively $y_d$) with $w_x$ (respectively $w_y$).

%
%%This is a set $\{x_1, \ldots , x_{d-1}\}$ such that $v\mapsto \{x_1\cdot v, x_2 \cdot v, \ldots , x_{d-1}\cdot v\}$ is injective. 
%Let $X$ be the span of these $\{x_1, x_2, \ldots x_{d-1} \}$ and $Y$ the span of the corresponding $\{y_1, y_2, \ldots y_{d-1}\}$. Recall that we know that $v\cdot x_i=\phi(v)\cdot y_i$ for each $i$ and for all $v$. It will be sufficient to find directions $x_d$ and $y_d$ not lying in $X$ or $Y$ respectively, such that $v\cdot x_d=\phi(v)\cdot y_d$ for all $v\in U$. We then could use the matrix argument above to show that $\phi$ is linear.
%
Let $\langle X_U \rangle$ and $\langle Y_{\phi(U)} \rangle$ denote the span of $X_U$ and let $w$ be perpendicular to $\langle X_U\rangle$. Note that $x_i\cdot w=0$ for $i=1,2, \ldots, d-1$ by construction. We claim that $w\cdot v\neq 0$ for all $v\in U$. To prove this suppose that $v_0\in U$ with $v_0\cdot w=0$. Since $U$ is open there exists $a,b>0$ such that both $av_0+bw$ and $av_0-bw$ are in $U$. However, $(av_0+bw)\cdot x_i=av_0\cdot x_i=(av_0+bw)\cdot x_i$ for $i=1,2, \ldots, d-1$ and this would contradict the injectivity from domain $U$ of the map $v\mapsto \{v\cdot x_1, v \cdot x_2, \ldots , v\cdot x_{d-1} \}$. As $U$ is an open connected subset and $w\cdot v$ does not vanish on $U$ we know either $v\cdot w>0$ for all $v\in U$ or $v\cdot w <0$ for all $v\in U$. By taking the negative, let $w_X$ denote the perpendicular to $\langle X_U\rangle$ such that $v\cdot w_X>0$ for all $v\in U$ if needed. 

Similarly we can define $w_Y$ as the unit vector orthogonal to $\langle Y_{\phi(U)}\rangle$ (the span of $Y_{\phi(U)}$) such that $w_Y\cdot \phi(v)>0$ for all $v\in U$.

Let $T:\langle X_U\rangle \rightarrow \langle Y_{\phi(U)}\rangle$ be the unique linear transformation which fixes the origin and such that $x_i \mapsto y_i$ for all $i$. Note that $v\cdot x=\phi(v)\cdot T(x)$ for each $x\in X$ and for all $v\in U$. 

Let $\pi_X$ and $\pi_Y$ be the orthogonal projections of $S^{d-1}$ onto $\langle X_U \rangle$ and $\langle Y_{\phi(U)} \rangle$ respectively. Projecting onto $\langle X_U\rangle$ and $\langle Y_{\phi(U)}\rangle$, we have
$$\pi_X(v) \cdot x = \pi_Y(\phi(v)) \cdot T x.$$ 

The adjoint $T^*$ of $T$ (with respect to the Euclidean inner products) satisfies
$$\pi_X(v) \cdot x = T^*(\pi_Y(\phi(v))) \cdot x$$
for all $x \in X$ and all $v \in U$ which in turn implies
$$\pi_X(A)= T^*\pi_Y(\phi(A))$$
for all $A\subset U$.

We will use  $\vol_{S^{d-1}}$, $\vol_{X}$ and $\vol_Y$ to denote the Lebesgue measures over $S^{d-1}$, $\langle X_U\rangle$ and $\langle Y_{\phi(U)}\rangle$. Since $T^*$ is linear there is a unique $\lambda_T>0$ such that $T^*$ scales volume by $\lambda_T$. For any $A\subset U$ we have 
%That is, for any $\subset X$ we have 
%$\vol_X(\pi_X(A))= T^*\pi_Y(\phi(A))$.
\begin{align}\label{eq:lambda}
\vol_{X}(\pi_X(A))=\lambda_T\vol_Y(\pi_Y(\phi(A))).
\end{align} 

 Let $A_\epsilon\subset U$ be an open subset of $U$ with diameter at most $\epsilon>0$. We will be comparing the measures of $A_\epsilon$, $\pi_X(A_\epsilon)$, $\phi(A_\epsilon)$ and $\pi_Y(\phi(A_\epsilon))$. If $v\in A_\epsilon$ then $$\vol_{S^{d-1}}(A_\epsilon)=\vol_{X}\pi_X(A_\epsilon)(v\cdot w_x+O(\epsilon)).$$ We similarly have $$\vol_{S^{d-1}}(\phi(A_\epsilon))=\vol_{Y}\pi_Y(\phi(A_\epsilon))(\phi(v)\cdot w_y+O(\epsilon)).$$
Since $\phi$ is measure preserving we know that $\vol_{S^{d-1}}(\phi(A_\epsilon))=\vol_{S^{d-1}}(A_\epsilon)$ and hence 
\begin{align}\label{eq:measure}
\vol_{X}\pi_X(A_\epsilon)(v\cdot w_X+O(\epsilon))=\vol_{Y}\pi_Y(\phi(A_\epsilon))(\phi(v)\cdot w_Y+O(\epsilon)).
\end{align}

Together \eqref{eq:lambda} and \eqref{eq:measure} imply that
\begin{align*}
\vol_{X}\pi_X(A_\epsilon)(v\cdot w_X+O(\epsilon))&=\vol_{Y}\pi_Y(\phi(A_\epsilon))(\phi(v)\cdot w_Y+O(\epsilon))\\
&=\lambda_T\vol_{X}\pi_X(A_\epsilon) (\phi(v)\cdot w_Y+O(\epsilon))
\end{align*}
 and hence $(v\cdot w_X+O(\epsilon))=\lambda_T(\phi(v)\cdot w_Y+O(\epsilon))$. By taking the limit as $\epsilon$ goes to zero we know $v\cdot w_X=\lambda_T\phi(v)\cdot w_Y$, with $\lambda_T$ is independent of $v\in U$. By setting $x_d=w_X$ and $y_d=\lambda_T w_Y$ we can state that $v\cdot x_d=\phi(v)\cdot y_d$ for all $v\in U$. This gives the $d$th linearly independent direction for the arguments above to imply that there is a linear map $M_U$ such that $\phi$ and $M_U$ agree on $U$. 

%Let $\phi^U$ be the linear extension of linear map of $\phi|_U$ to the entire domain of $\R^{d}$.
Since $M_U$ maps preserves distances for all $v\in U$ we can apply Lemma \ref{lem:orthogonal} to conclude that $M_U\in O(d)$.

Since $\phi$ is definable we have a stratification of $S^{d-1}$ whose lower strata consist of $W(X_1)\cup \phi^{-1}(W(X_2))$. 
Set $V$ be the set of the $(d-1)$ strata, $E$ the $(d-2)$ strata, and $f$ be the map which realises each element of $V$ or $E$ as its corresponding subset in $S^{d-1}$. Let $G$ be the graph with vertex set $V$ and edges $E$  where the endpoints of edge $e$ are the vertices $v$ such that $f(e)\subset \overline{f(v)}$. We know from Lemma \ref{lem:connected} that $G$ is connected.

Consider adjacent $d-1$ dimensional strata $U_1$ and $U_2$ which are adjacent in $G$. This means that there is a $d-2$ strata within $\partial U_1\cap \partial U_2$. As a $d-2$ dimensional subset of $S^{d-1}$ we know that it must contains at least $d-1$ linearly independent points $\{v_1, v_2, \ldots v_{d-1}\}\subset \partial U_1\cap \partial U_2$. Since $\phi$ is continuous we know that $M_{U_1}(v_i)=M_{U_2}(v_i)$ for all $i$.  Choose $w\in S_{d-1}$ such that $w$ is perpendicular to all the $v_i$ (note that there are exactly two options). Since $\phi^{U_1}, \phi^{U_2}\in O(d)$ we can observe that $M_{U_1}(w)$ is perpendicular to $M_{U_1}(v_i)$ and $M_{U_2}(w)$ is perpendicular to $M_U{U_2}(v_i)$ for all $i$. As $M_U{U_1}(v_i)=\phi^{U_2}(v_i)$ for all $i$, this implies that $M_U{U_2}(w)=\pm M_U{U_1}(w)$. 

Choose $v\in \partial U_1\cap \partial U_2$ in the space of the $\{v_1, v_2, \ldots v_{d-1}\}$, and $a,b\neq 0$ such that $av-bw\in U_1$ and $av+bw\in U_2$.  If $M_{U_2}(w)=- M_{U_1}(w)$ then we have
\begin{align*}\phi(av-bw)&=M_{U_1}(av-bw)\\
&=a M_{U_1}(v) -b M_{U_1}(w)\\
&=aM_{U_2}(v) +bM_{U_2}(w)\\
&=M_{U_2}(av+bw)\\
&=\phi(av+bw)
\end{align*}
which contradicts the injectivity of  $\phi$. 
This implies that $M_{U_2}(w)=M_{U_1}(w)$ and thus $M_{U_2}=M_{U_1}$.

Since $G$ is connected and $M_{U_2}=M_{U_1}$ for $U_1$ and $U_2$ adjacent vertices in $G$, we conclude that $M_U$ must be the same for all $U$, which we denote $M$, and that $M\in O(d)$. We have two continuous maps $\phi, M|_{S^{d-1}}:S^{d-1} \to S^{d-1}$ that agree on an open dense subset (the union of the top strata) and hence must be the same over the entire domain. This implies that $\phi \in O(d)$ (by which we mean $\phi$ is the restriction of some $M\in O(d)$ to the unit sphere).

Now recall that by construction $\phi=\ECT(K_2)^{-1}\circ \ECT(K_1)$, so
\[
	\ECT(K_2)(\phi(v)) = \ECT(K_1)(v) \qquad \forall v\in S^{d-1}. 
\]
Since $\phi \in O(d)$ we can apply it to $K_1$ itself viewed as a subset of $\R^d$.
Note that in this case we have the obvious formula 
\[
\ECT(K_1)(v)=\ECT(\phi(K_1))(\phi(v)) \qquad \forall v\in S^{d-1}.
\]
Since both $\ECT$ and $\phi$ are injective we then have our desired implication.
\[
	\ECT(K_2)(\phi(v)) = \ECT(\phi(K_1))(\phi(v)) \quad \forall v\in S^{d-1} \qquad \Rightarrow \qquad K_2=\phi(K_1)
\]
\end{proof}

%%   %%%%%%%%  %%%%%%%%  %%%%%%%%  %%%%%%%%%%
%%   %%        %%        %%            %%
%%   %%%%%%%%  %%%%%%%%  %%            %%  
%%         %%  %%        %%            %%
%%   %%%%%%%%  %%%%%%%%  %%%%%%%%      %%

\section{The Sufficiency of Finitely Many Directions}\label{sec:finite}

In this section we reach the main result of our paper, which provides the first finite bound (to our knowledge) on the number of ``inquiries'' required to determine a shape belonging to a particular uncountable set of shapes, which we call $\mathcal{K}(d,\delta,k_{\delta})$.
These ``inquiries'' do not yield yes or no responses, but rather an Euler curve or persistent diagram.
% At a high level the meaning of the three parameters used to carve out this class of shapes are as follows:
The set of shapes belonging to $\mathcal{K}(d,\delta,k_{\delta})$ can be described at a high level as follows:
\begin{enumerate}
	\item[(1)] A ``shape'' is an embedded geometric simplicial complex $K\subset \R^d$. Two different embeddings of the same underlying combinatorial complex $K$ are regarded as different shapes.
	\item[(2)] Each embedded complex $K$ is not ``too flat'' near any vertex. This guarantees that there is some set of directions with positive measure in which a non-trivial change of Euler characteristic (or homology) at that vertex can be observed. This measure is bounded below by a parameter $\delta$, see Definition~\ref{defn:delta-observable}.
	\item[(3)] No complex $K$ has too many critical values when viewed in any given $\delta$-ball of directions. Of course, for a fixed finite complex the number of critical values gotten by varying the direction is certainly bounded above, but we assume a uniform upper bound $k_{\delta}$ for our entire class of shapes $\mathcal{K}(d,\delta,k_{\delta})$.
\end{enumerate}

The proof of Theorem~\ref{thm:finitenum} proceeds in two steps.
First, we show that the vertex locations of any member $K\in \mathcal{K}(d,\delta,k_{\delta})$ can be determined simply by measuring changes in Euler characteristic when viewed along a fixed set of finitely many directions $V$.
Once these vertices are located, the associated hyperplane division of the sphere described in Section~\ref{sec:stratified} is then determined.
Since we can provide a uniform bound on the number of vertices of any element of $\mathcal{K}(d,\delta,k_{\delta})$, we then have a bound on the total possible number of top-dimensional strata of the sphere determined by hyperplane division.
The second step is to then sample a direction from each individual top-dimensional stratum.
Since Proposition~\ref{prop:deduce} guarantees that we can interpolate the Euler Characteristic Transform over all of the top-dimensional strata, continuity guarantees that we can determine the entire transform using only these sampled directions.
Finally, since Theorem~\ref{thm:ECT-injects} implies that $\ECT(K)$ uniquely determines $K$, we then obtain the fact that any shape in $\mathcal{K}(d,\delta,k_{\delta})$ is determined by finitely many draws from the sphere.

The mathematical ideas underlying Theorem~\ref{thm:ECT-injects} were central in the development of an algorithm that takes as input 
two classes of shapes and highlights the physical features that best describe the variation between them without requiring correspondences \cite{Wang} and extended to characterize physical differences between classes of proteins in \cite{Wai}. Again, we would like to make clear that the theory developed in this section is relevant to practical imaging applications.

As one might imagine, the above argument rests on many intermediary technical propositions and lemmas.
The first lemma we introduce is an application of the generalized pigeonhole principle, which will be used to pin down the locations of a vertex set of a fixed embedded simplicial complex.

\begin{lem}\label{lem:pigeon}
Let $X=\{x_1,\ldots, x_{l}\}\subset \R^d$ be a finite set of points and
let $D=\{v_1,\ldots,v_n\}\subset S^{d-1}$ be a set of directions in general position.
We let $H(v_i,x_j)$ denote the hyperplane with normal vector $v_i$ that passes through point $x_j$.
Suppose $n$ of the hyperplanes in $\mathcal{H}=\{H(v,x) \mid v\in D, x\in X\}$ intersect at $y\in \R^d$.
If $n > (d-1)l$, then $y\in X$.
\end{lem}

\begin{proof}
The hypothesis implies that the set of hyperplanes $\mathcal{H}$ has at least $n$ elements, because we've assumed that $n$ of them intersect at some point $y\in \R^d$.
Among these $n$ hyperplanes consider the assignment of the plane $H(k)$ to some $x\in X$ with $x\in H(k)$.
This defines a map from an $n$ element set to the $l$ element set $X$.
Since $n\geq (d-1)l +1$, the generalized pigeonhole principle implies at least one element of $X$ is mapped to by $d$ different hyperplanes, i.e.~at least one element $\tilde{x}\in X$ is contained in $d$ of the $n$ hyperplanes containing $y$.
Denote the normal vectors of these $d$ hyperplanes by $v_{i_1},\ldots,v_{i_d}$.
Note that since $\tilde{x}$ and $y$ are contained in these $d$ hyperplanes we have the following system of equations
\[
	v_{i_k}\cdot \tilde{x} = v_{i_k}\cdot y \qquad \text{for} \qquad k=1,\ldots,d.
\]
Now we invoke the general position hypothesis, which says that the $\{v_{i_k}\}$ are linearly independent, and we deduce that $\tilde{x}=y$.
\end{proof}

We now consider the two types of information that our two topological transforms can observe when looking in a direction $v\in S^{d-1}$.
As has been the pattern for this paper, we first consider changes for the Euler Characteristic Transform and then deduce results about the Persistent Homology Transform.

\begin{defn}
Let $K\subset \R^d$ be a finite simplicial complex and $v\in S^{d-1}$. 
A vertex $x\in K$ is \define{Euler observable} in direction $v$ if $\ECT(K)(v)$ changes value at $x\cdot v$. 
Given a subset of directions $A\subset S^{d-1}$, then a vertex $x\in X$ is Euler observable in $A$ if $x$ is observable for some $v\in A$.
We will often say \define{observable} when the context is clear.	
\end{defn}

For the purposes of sampling the Euler Characteristic Transform, it is important to guarantee that each vertex is observable for some positive measure of directions.
We make this precise by quantifying the above observability criterion via a real parameter $\delta$.

\begin{defn}\label{defn:delta-observable}
Let $K\subset \R^d$ be a geometric simplicial complex and let $x\in K$ be a vertex of $K$.
We say the vertex $x$ is at least \define{$\delta$-Euler observable} if there exists a ball of directions $B(v_x,\delta) \subset S^{d-1}$ such that the Euler characteristic of the sublevel sets of $h_v$ changes at $v\cdot x$ for all $v\in B(v_x,\delta)$.
\end{defn}

The idea of a $\delta$-observable ball of directions at a vertex $x$ is that the local geometry of the simplicial complex around $x$ should not be ``too flat'' or similar to a hyperplane.
This is indicated in the next example.

\begin{ex}
Let $K$ be a triangle in $\R^2$ with vertices $\{x,y,z\}$. The vertex $y$ is observable in direction $v$ exactly when $v\cdot y$ is smaller that both $v\cdot x$ and $v\cdot z$. If the angle at $y$ is $\theta$ then $y$ is $\delta$-observable for all $\delta \leq \pi-\theta$.
%Let $K$ be the triangulation of a closed disk $\mathbb{D}$ that is embedded in a hyperplane $H$, defined by a normal vector $v$. 
%If $x$ is a vertex in the interior of $\mathbb{D}$, then it is observable in the direction $v$ (and $-v$), but it is not observable in any other direction. Thus it is not $\delta$-observable for any $\delta>0$. 
\end{ex}

This allows us to make precise our definition of $\mathcal{K}(d, \delta, k_\delta)$.

\begin{defn}\label{defn:shape-class}
	Let $\mathcal{K}(d,\delta,k_{\delta})$ denote the set of all embedded simplicial complexes $K$ in $\R^d$ with the following two properties:
	\begin{enumerate}
		\item[(1)] Every vertex $x\in K$ is at least $\delta$-observable.
		\item[(2)] For all $v\in S^{d-1}$, the number of $x\in X$ for which $h_w(x)$ is an Euler critical value of $h_w$ for some $w\in B(v,\delta)$ is bounded by $k_{\delta}$.
	\end{enumerate}  
\end{defn}

The second condition imposes a uniform upper bound on the number of critical points for all directions in $\delta$ neighborhood of direction $v$ over all directions $v$ on 
the sphere.

Before stating the true importance of the $\delta$-observable condition, we remind the reader of a definition.
\begin{defn}
A \define{$\delta$-net} in a metric space $(M,d)$ is a subset $N$ of points such that the union of $\delta$-balls about these points $\cup_{p\in N} B(p,\delta)$ covers the entire space $M$.
\end{defn}

\begin{lem}
Let $K\subset \R^d$ be a geometric simplicial complex.
If a vertex $x\in K$ is $\delta$-observable for some direction $v_{x} \in S^{d-1}$, then for any $\delta$-net $V\subset S^{d-1}$ of directions there must be a direction $v\in V$ which can observe $x$.
\end{lem}

\begin{proof}
Since $V$ is a $\delta$-net there must be a $v\in V$ whose $\delta$-ball includes $v_{x}$. 
Since $x$ is $\delta$-observable for $v_x$, it's observable for $v$ as well.
\end{proof}

In order to extend the pigeonhole principle argument of Lemma~\ref{lem:pigeon} to find vertices that are $\delta$-observable, we need a covering argument.
This rests on the following technical lemma.

\begin{lem}\label{lem:ballvol}
Let $r<\frac{\pi}{2}$ and fix a $w\in S^{d-1}$. 
We write $B_{S^{d-1}}(w,r) = \{v\in S^{d-1}: d(v,w)<r\}$ for the ball of directions about $w$ of radius $r$.
Recall that $\mu$ denotes the Lebesgue measure on $S^{d-1}$ and $\omega_{d-1}$ is the Lesbesgue measure for the unit $(d-1)$-ball in $\R^{d-1}$.
The following string of inequalities hold:
$$\omega_{d-1} \sin(r)^{d-1} \leq \mu(B_{S^{d-1}}(w,r) )\leq \omega_{d-1} r^{d-1}$$
\end{lem}
\begin{proof}
The second inequality compares balls in the sphere to balls in Euclidean space. Since the sphere is positively curved the volume of a ball in the sphere is less than a ball of the same radius in Euclidean space of the same dimension. 

The first inequality compares the volume of the $d-1$ dimensional ball in Euclidean space with radius $r$ around $w$ to the volume of the projection onto the tangent plane at $w$. The projection has radius $\sin(r)$.
\end{proof}

By knitting together the $\delta$-observable condition with a uniform bound on the number of critical values in any ball of radius $\delta$, we obtain an upper bound on the number of vertices.

\begin{lem}\label{lem:boundobservable}
The total number of vertices $X$ for any $K\in \mathcal{K}(d,\delta,k_\delta)$ is bounded by
\[
	|X| \leq \frac{dk_\delta}{\sin(\delta/2)^{d-1}}.
\] 
We know that $|X|=O(\frac{dk_\delta}{\delta^{d-1}})$.
\end{lem}

\begin{proof}
% Let $X$ denote the vertex set of $K\in \mathcal{K}(d,\delta,k)$.
The proof follows from an upper bound in the number of $\delta$-balls needed to cover $S^{d-1}$. To do this we will use the relationship between packing and covering numbers. Let $\omega_{d-1}$ denote the volume of a ball in $\R^{d-1}$.
From Lemma \ref{lem:ballvol} we know that the area of a ball of radius $\delta/2$ within $S^{d-1}$ is bounded below by $\omega_{d-1} \sin(\delta/2)^{d-1}$. We also know that the surface area of $S^{d-1}$ is $d\omega_{d-1}$

This implies that the maximum number of balls of radius $\delta/2$ that can be packed into $S^{d-1}$ is bounded above by $$\frac{d\omega_{d-1}}{\omega_{d-1} \sin(\delta/2)^{d-1}}=\frac{d}{\sin(\delta/2)^{d-1}}.$$

Using the standard relationship that the covering number at radius $\delta$ is bounded above by the packing number at radius $\delta/2$ we know that there exists a covering of $S^{d-1}$ by $\frac{d}{\sin(\delta/2)^{d-1}}$ balls of radius $\delta$. In each of these balls in the cover at most $k_\delta$ different vertices can by observed. This implies that $$|X|\leq \frac{dk_\delta}{\sin(\delta/2)^{d-1}}.$$

\end{proof}

% \color{purple}
\begin{defn}\label{defn:delta-C-net}
A  $(\delta,C)$-net is a subset of directions that meets every $\delta$-ball in a set of cardinality $C$.
\end{defn}

% \color{black}

\begin{prop}\label{prop:find_vert}
For $K \in  \mathcal{K}(d,\delta,k_\delta)$ and the following constant which does not depend on $K$
$$C=C(d,\delta,k_\delta)=(d-1)\, k_\delta +1.$$
If $V$ is a $(\delta,C)$-net in $S^{d-1}$ with the vectors in $V$ in general position, 
then we can determine the location of the set of vertices of $K$ using only the Euler curves generated from the directions in $V$.
\end{prop} 

\begin{proof}
The proof consists of first stating a constructive algorithm that will specify the location of all the vertices of $K$ given Euler characteristic curves and then checking the correctness of the algorithm. 

The algorithm declares a point $x$ to be a vertex of $K$ whenever there exists a set of different directions $V_x$, with the number of directions $|V_x|\geq C$ and $\text{diam}(V_x)\leq 2\delta$, such $v\cdot x$ is a critical value for height function $h_v$ for all $v\in V_x$. 
To check the correctness of this algorithm we need to check that all the vertices are found and that no extra points are declared.

We first show that the algorithm will include every vertex in $K$. 
Consider a vertex $x$ in $K$, since $x$ is $\delta$-Euler observable there is some ball $B(w,\delta)$ such that $x$ is observed from every direction in $B(w,\delta)$. 
Let $V_x=V\cap B(w,\delta)$. 
By construction $\text{diam}(V_x)\leq \delta$ and $v\cdot x$ is a critical value for height function $h_v$ for all $v\in V_x$. 
Since $V$ is $(\delta,C)$-net we know $|V_x|\geq C$. Thus our algorithm will declare $x$ to be a vertex in $K$.

We now show that the algorithm will not declare any extra points. 
Suppose that $\tilde{x}$ 
is a point declared by our algorithm to be a vertex of $K$. 
This implies that there exists a vertex set $V_{\tilde{x}}$ 
and a ball of directions $B(w,\delta)$ such that both 
$V_{\tilde{x}} \subset B(w,\delta)$ and 
$$|V_{\tilde{x}}|\geq (d-1) \, k_\delta  +1$$ 
and $\tilde{x} \cdot v$ 
is a critical value of the height function $h_v$ for every $v\in V_{\tilde{x}}$. 
By assumption, we know that the number of observable vertices of $K$ for any $v\in B(w, \delta)$ is at most 
$k_\delta$.
We then can apply our pigeonhole lemma, Lemma~\ref{lem:pigeon}, with
\[
n=|V_{\tilde{x}}| \qquad \text{and} \qquad l=k_\delta 
\]
to observe that $\tilde{x}$ must be one of these observable vertices of $K$ and hence a vertex of $K.$
\end{proof}

We now expand the above definitions and arguments to the homology setting.

\begin{defn}
A vertex $x$ in a geometric simplicial complex $K\subset \R^d$ is \define{homology observable} if there is a direction $v\in S^{d-1}$ so that $\PHT(K)(v)$ contains an off diagonal point with either birth or death value $v\cdot x$.
Moreover, we say that a vertex $x\in K$ is at least \define{$\delta$-homology observable}, if it is homology observable for every direction in some ball $B(v_x,\delta) \subseteq S^{d-1}$ of directions.
\end{defn}

The following statement is the analog to Proposition \ref{prop:find_vert} for the Persistent Homology Transform, so we omit the proof as it is identical.

\begin{cor}\label{prop:find_vert1}
Suppose $K \subset \R^d$ is a finite simplicial complex such that every $K$ is $\delta$-homology observable with the same critical value conditions as in Proposition \ref{prop:find_vert}. Let $V$ be a $(\delta,C)$-net with the vectors in $V$ all in general position.
The location of the set of vertices of $K$ can be determined using only the persistence diagrams generated from the directions in $V.$
\end{cor} 

We now state the main theorem of the paper.

\begin{thm}\label{thm:finitenum}
Any shape in $\mathcal{K}(d,\delta,k_\delta)$ can be determined using the $\ECT$ or the $\PHT$ using no more than
% \color{purple}
$$\Delta(d,\delta,k_\delta) = \left((d-1) \, k_\delta \,  +1 \right) \left(1+\frac{3}{\delta}\right)^d+O\left(\frac{d^{d+1}k_\delta^{2d}}{\delta^{2d(d-1)}}\right)$$
% \color{black}
directions, where $\delta$ and $k_\delta$ are determined by the Euler or homological critical values, respectively.
\end{thm}

\begin{proof}
We specify the proof for the ECT case, since the PHT case is identical.

The proof consist of showing that given a finite set of vectors $v_i \in V$ and the corresponding transform values $\ECT(K,v_i)$, we can recover $\ECT(K,v)$ for any direction $v\in S^{d-1}$.
The proof proceeds as follows, after fixing some $K\in \mathcal{K}(d,\delta,k_\delta)$.

\begin{enumerate}
\item[(1)] Set $V$ to be a general position $(\delta,C)$-net in $S^{d-1}$. Here $C$ is the constant from
Proposition \ref{prop:find_vert}, $C=(d-1)\, k_\delta +1.$
\item[(2)] Proposition~\ref{prop:find_vert} also states that the locations of the vertices $X$ of the initial simplicial complex $K$ can be recovered from the Euler characteristic curves generated from the directions in V.
\item[(3)] Now that we have the vertex set $X$ of $K$, consider the hyperplane division $W(X)$ of $S^{d-1}$ and its complement $\Sigma(X)$.
\item[(4)] For each connected component $U_j$ of $\Sigma(X)$ pick a direction $u_j \in U_j$ and evaluate $\ECT(K,u_j)$. 
By Lemma \ref{lem:KxU} we  can deduce the 
 $\ECT(K,w)$ for any $w\in U_j$ from $\ECT(K,u_j)$.
 \item[(5)] Given step (4), by continuity we know $\ECT(K,v)$ for $v\in W(X)$.
\end{enumerate}
\vspace{.1in}

% \color{purple}
We now specify a bound on the number of directions in the set $V$, which rests on a bound for a $(\delta,C)$-cover of the unit sphere $S^{d-1}$.
We first state how to construct a $(\delta,C)$-net. 
Consider the points $V$ as the centers of a collection of $C$ number $\delta'$-nets with $\delta'=\frac{2}{3}\delta$ (although any $\delta'<\delta$ will suffice). Now perturb all the points in $V$ slightly to put them into general position. 
The result of this procedure is a $(\delta,C)$-net $V$ with points in general position.
 By Lemma 5.2 in \cite{Versh}, the bound on the number of points needed for a $\delta'$-net is $N = \left(1+\frac{2}{\delta'}\right)^d$.
 This then proves the maximum number of directions required to specify the vertex set $X$ for any element $K\in \mathcal{K}(d,\delta,k_\delta)$ is 
$$|V| \leq \left((d-1)\, k_\delta +1 \right) \left(1+\frac{3}{\delta}\right)^d.$$
% \color{black}
Finally, we bound the number of directions required by Step (4) above.
We note that $W(X)$ is the union of $n(X)=\binom{|X|}{2}$ hyperplanes.
This divides the sphere up into at most
\[
\sum_{j=0}^d \binom{n(X)}{j} \leq d\left(\frac{e n(X)}{d}\right)^d,
\]
note the above computation is standard computation used in both the uniform law of large numbers of sets \cite{VC} as well as the study of range spaces in discrete geometry 
\cite{AHW}. Recalling Lemma~\ref{lem:boundobservable} and using the fact that $n(X)= O(|X|^2)$, we see that the number of top dimensional strata is 
\begin{align*}
O\left(d\left(\frac{n(X)}{d}\right)^d\right)&=O\left(\frac{|X|^{2d}}{d^{d-1}}\right)\\
&=O\left(\frac{d^{d+1}k_\delta^{2d}}{\delta^{2d(d-1)}}\right).
\end{align*}
This completes the proof of the theorem.
\end{proof}

%%   %%%%%%%%  %%%%%%%%  %%%%%%%%  %%%%%%%%%%
%%   %%        %%        %%            %%
%%   %%%%%%%%  %%%%%%%%  %%            %%  
%%         %%  %%        %%            %%
%%   %%%%%%%%  %%%%%%%%  %%%%%%%%      %%

\section{Future Directions}

We conclude with a discussion of future directions that one might explore. We believe these questions would be intriguing to answer and important for geometry and statistics. 

\begin{itemize}
\item In Sections~\ref{sec:ECT-injects} and~\ref{sec:PHT-injects} we used the class of compact, o-minimal sets to state our results. 
The primary purpose of this choice was to avoid inconsistencies between theory involving the two notions of Euler characteristic---the ordinary and compactly-supported/definable Euler characteristics. For general sets we can not read off the definable Euler characteristic curve from the persistent homology. One direction of inquiry is considering a variant of persistent homology that naturally decategorifies to compactly-supported Euler characteristic. If we are restricted to filtrations only including locally compact sets then a natural choice would be a version of persistent Borel-Moore homology.
It may be possible that we could use the full generality of Schapira's inversion theorem to prove the injectivity of the PHT for any locally compact o-minimal sets. 
One challenge with considering Borel-Moore persistent homology is that Borel-Moore homology is functorial with respect to proper continuous maps, but not to arbitrary continuous maps. 
When we work with non-compact definable sets, the inclusion maps of sublevel sets are continuous but are not proper continuous in general. 
This means we do not get the persistence module structure maps needed here. Another complication is that not all o-minimal sets are locally compact and little is defined, let alone understood, about Borel-Moore homology for sets that are not locally compact.
% \color{red}
% Please look over. Tried to implenent the advice in comment 33.
% \color{black}

\item In Section~\ref{sec:stratified} we gave a complete description of the ECT and the PHT of finite simplicial complexes in terms of a stratified space decomposition of the sphere of directions, given by hyperplanes.
We then showed that one can interpolate the ECT by only knowing one curve from each stratum alongside knowing the vertex set. Both of these steps have natural analogs for shapes that aren't cut out by linear inequalities, but instead are determined by algebraic ones. How might one decompose the sphere and interpolate the ECT for shapes cut out using quadratic equations, for example?

\item Continuing the above question to Section~\ref{sec:affine}, one might ask whether more general constructible/semialgebraic sets have pushforward measures that are invariant to actions by $O(d)$.

\item Finally, we note that Section~\ref{sec:finite} used piecewise linearity of the involved shapes in an essential way.
There is no direct way to extend these results for arbitrary constructible sets, but there are a variety of related problems. One option is to explore whether it extends for some well-behaved family of shapes. 
For example, we could ask whether there is such a finiteness result for smooth algebraic manifolds with lower bounds on curvature. Another direction is a reconstruction up to some small error. One might be able to prove that given a set of directions, that we can construct a shape whose Euler curves agree with those for that sample direction such which must be close to the unknown shape with high probability. 
\end{itemize}

\section*{Acknowledgments}
The authors would like to thank the Institute for Mathematics and its Applications (IMA) at the University of Minnesota for hosting the authors during numerous visits. This work was completed at the Special Workshop on Bridging Statistics and Sheaves (SW5.21-25.18).
The authors would also like to thank Ingrid Daubechies, Doug Boyer, Tingran Gao, Robert Ghrist, Yuliy Baryshnikov, Lorin Crawford, Henry Kirveslahti, Yusu Wang, and John Harer for helpful discussions. 
JC would like to thank Kathryn Hess and EPFL for supporting his travel to collaborate with KT on this project. KT is the recipient of an Australian Research Council Discovery Early Career Award (project number DE200100056) funded by the Australian Government. 
JC would also like to thank NSF CCF-1850052 and NASA Contract
80GRC020C0016 for supporting his research.
Finally, SM would also like to thank partial funding from the Human Frontier Science Program, as well as NSF DMS 17-13012, NSF ABI 16-61386, and NSF DMS 16-13261.

Finally, the authors would also like to thank an anonymous referee who made substantial contributions to improving the article, including catching mistakes in earlier drafts and suggesting ways to fix them.
The result of this implicit collaboration is a much stronger paper that makes explicit several results that were only hinted at in earlier versions of the paper.

\newpage

\bibliographystyle{amsalpha}

\end{document}